\documentclass[10pt, a4paper]{amsart}
\usepackage[dvips]{epsfig}
\usepackage{amsmath}
\usepackage{amssymb}
\usepackage{amsfonts}
\usepackage{dsfont}
\usepackage{bbm}
\usepackage{amsthm}
\usepackage{amsbsy}
\usepackage{amsgen}
\usepackage{amscd}
\usepackage{amsopn}
\usepackage{amstext}
\usepackage{amsxtra}
\usepackage[utf8]{inputenc}
\usepackage{enumerate}
\usepackage{hyperref}
\usepackage{lineno}
\usepackage{marginnote}
\usepackage{verbatim}
\usepackage{color}
\usepackage{calrsfs}
\usepackage{times, graphicx}
\usepackage{eso-pic}
\usepackage{lastpage}
\usepackage{geometry}
\DeclareMathAlphabet{\pazocal}{OMS}{zplm}{m}{n}
\usepackage[foot]{amsaddr}
\usepackage{geometry}

\newtheorem{theorem}{Theorem}[section]
\newtheorem{lemma}[theorem]{Lemma}

\theoremstyle{definition}

\newtheorem{proposition}[theorem]{Proposition}
\newtheorem{corollary}[theorem]{Corollary}

\theoremstyle{remark}

\numberwithin{equation}{section}

\allowdisplaybreaks

\begin{document}

	\title[No smooth phase transition for the Nodal Length of band-limited
	spherical random fields]{No smooth phase transition for the Nodal Length \\of band-limited
		spherical random fields}

	\author{Anna Paola Todino\textsuperscript{}}
	\email{annapaola.todino@uniroma1.it}
	\address{\textsuperscript{}Department of Statistics, Sapienza University of Rome\\}

	\date{}
	
	\subjclass[2020]{60G60, 42C10, 33C55,  62M15, 35P20}

	\keywords{Gaussian Eigenfunctions, Spherical Harmonics, Berry's cancellation, band-limited functions}

	\begin{abstract}
	In this paper, we investigate the variance of the nodal length for
	band-limited spherical random waves. When the frequency window includes a number of
	eigenfunctions that grows linearly,  the variance of the nodal length is
	linear with respect to the frequency, while it is logarithmic when a
	single eigenfunction is considered. Then, it is natural to conjecture
	that there exists a smooth transition with respect to the number of
	eigenfunctions in the frequency window; however, we show here that the
	asymptotic variance is logarithmic whenever this number grows
	sublinearly, so that the window "shrinks". The result is achieved by exploiting the
	Christoffel-Darboux formula to establish the covariance function of the
	field and its first and second derivatives. This allows us to compute
	the two-point correlation function at high frequency and then to derive
	the asymptotic behaviour of the variance.
	
	\end{abstract}
\maketitle
	
\newcommand{\ABS}[1]{\left(#1\right)} 
\newcommand{\veps}{\varepsilon} 

\section{Introduction}
\subsection{Laplacian eigenfunctions and Berry's  random wave model}

Let $(\mathcal{M},g)$ be a compact Riemannian manifold and $\Delta_g$ be the Laplace-Beltrami operator on $\mathcal{M}$. We denote the corresponding eigenvalues $\{\lambda_j\}_{j\geq0}$ with associate orthonormal basis $L^2(\mathcal{M})$ consisting of eigenfunctions $\{f_{\lambda_j}\}_{j\geq0}$, i.e.  
\begin{equation}\label{laplacian}
 \Delta_g f_{\lambda_j}+ \lambda_j f_{\lambda_j}=0.
 \end{equation}

The nodal set of  $f_\lambda$ is its zero set

$$f_\lambda^{-1}(0):=\{x \in M: f_\lambda (x)=0\}.$$

It is known that $f_\lambda^{-1}$ is the union of smooth curves outside a finite set of points \cite{Cheng}. Yau conjectured that its volume (the nodal length in dimension 2) satisfies 

\begin{equation}
\label{Yau}
c\sqrt{\lambda} \leq Vol(f_{\lambda}^{-1}(0))\leq C\sqrt{\lambda}
\end{equation}
for some constants $c,C>0$ which depend on $\mathcal{M}$ only. Donnelly and Fefferman \cite{DF} showed that Yau's conjecture holds for any real-analytic
manifold (of any dimension), and recently, Logunov and Malinnikova \cite{Log18,Log,LM1} proved
the lower-bound in the smooth case and gave a polynomial upper-bound. 
The studies of Laplacian eigenfunctions have been recently intensified in view of the Berry's conjecture \cite{Ber77}. Indeed, according to
Berry, the behaviour of high energy ($\lambda \to \infty$) \emph{deterministic} eigenfunctions should be universal, at least on generic chaotic surfaces, meaning that one can compare such an eigenfunction $f_{\lambda}$ of large eigenvalue $\lambda$ with the random
monochromatic plane wave of wavelength $\sqrt{\lambda}$, that is, the (unique in law) centered Gaussian random field on $\mathbb R^2$ whose covariance kernel is 
\begin{equation}\label{berrymodel}
J_0(\sqrt{\lambda} \| x - y \|), \qquad x,y\in \mathbb R^2,
\end{equation}
$J_0$ being the Bessel function of the first kind of order zero.
Indeed Berry's Random Wave model (RWM) is the centred Gaussian random field $u: \mathbb{R}^2 \to \mathbb{R}$ of monochromatic isotropic waves, described by the covariance function 
$$r(x,y)=r(x-y):=\mathbb{E}[u(x)\cdot u(y)]=J_0(|x-y|).$$

Berry's RWM has recently been investigated by many authors, we refer to \cite{Vidotto}, \cite{NPV}, \cite{DNPR}, 	\cite{NPR} for some recent studies on nodal lengths and	nodal sets.
	
\subsection{Random Spherical harmonics} Among models of the form (\ref{laplacian}) a lot of interest has been aroused by the case of random spherical eigenfunctions. Motivations are mainly two: on the one hand Berry's conjecture, indeed random spherical eigenfunctions exhibit the same covariance structure as
the Euclidean case, in the scaling limit, and on the other hand its application to the Cosmic Microwave Background \cite{M e Peccati}. We describe the model below.

Let us consider  $\mathcal{M}=\mathbb{S}^2$ and then the Helmholtz equation
\begin{equation}\label{Helmotz}
\Delta_{\mathbb{S}^2} T+\lambda_\ell T=0,
\end{equation}

where $\Delta_{\mathbb{S}^2} $ is the spherical Laplacian and $\lambda_\ell =\ell(\ell+1)$ for $\ell \in \mathbb{N}$. Spherical eigenfunctions are defined as the solutions of (\ref{Helmotz}).
For any eigenvalue $-\lambda_\ell$ we can choose an arbitrary $L^2-$orthonormal basis $\{Y_{\ell m}(\cdot)\}_{m=-\ell,\dots,\ell}$ and consider random eigenfunctions of the form 
\begin{equation} \label{ttt}T_\ell(x)=\sum_{m=-\ell}^\ell a_{\ell m} Y_{\ell m}(x), 
\end{equation}

where  $\{a_{\ell m}\}$ is the array of random spherical harmonic coefficients, which are independent, safe for the condition $\bar{a}_{\ell m} = (-1)^m a_{\ell,-m};$ for $m \ne0$ they are standard complex-valued Gaussian variables, while $a_{\ell 0}$ is a standard real-valued Gaussian variable. They satisfy
\begin{equation*}
\mathbb{E}[a_{\ell m}\bar{a}_{\ell ^{\prime }m^{\prime }}]=\delta _{\ell
}^{\ell ^{\prime }}\delta _{m}^{m^{\prime }}.
\end{equation*}

The random fields $\{T_\ell(x):  x\in \mathbb{S}^2 \}$ are isotropic centred Gaussian with covariance function given by
$$\mathbb{E}[T_\ell(x)T_\ell(y)]= \frac{2\ell+1}{4\pi}P_\ell(\cos d(x,y)) \quad \mbox{ and then } \quad {\rm Var}(T_{\ell}(x))=\frac{2\ell+1}{4\pi},$$
denoting with $P_\ell$ the Legendre polynomial and $d(x,y)=\arccos \langle x,y\rangle$ the geodesic distance on the sphere, see \cite{M e Peccati} for more details.
The nodal set of $T_\ell$ is $$T_\ell^{-1}(0)=\{ x \in \mathbb{S}^2: T_\ell(x)=0 \}.$$
The analysis of its volume, i.e. the nodal length, and other properties in the high energy domain have been largely investigated for example in \cite{NS09}, \cite{NS17}, \cite{W09}, \cite{W}, \cite{MRW2017}.

It is worth stressing that the \emph{scaling limit} of random spherical harmonics is Berry's RWM (\ref{berrymodel}). Indeed, Hilb's asymptotic formula, see \cite{szego} 
Equation 8.21.7 and \cite{W}, ensures that, for every $\epsilon >0$, uniformly for $\theta\in [0,\pi-\epsilon]$,
\begin{equation}\label{hilb}
P_\ell(\cos \theta) \sim \sqrt{\frac{\theta}{\sin \theta}} J_0(\ell\theta), \qquad \ell \to +\infty.
\end{equation}
When valid, Berry's ansatz implies Yau's conjecture in (\ref{Yau}) (see \cite{Yau82}).

Berry's conjecture goes far beyond the macroscopic setting, i.e. the random wave model (\ref{berrymodel}) is applicable to \emph{shrinking domains}. For example, as it has been shown in \cite{Todino2}, it should be a good model for $f_{\lambda}^{-1}(0) \cap B_{r_\lambda}$, in particular for the nodal length
 lying inside a shrinking geodesic ball $B_{r_\lambda}$ of radius slightly above the Planck scale: $r_\lambda \approx \frac{C}{\sqrt{\lambda}}$ with $C\gg 0$ sufficiently big. 
For a recent survey on nodal structures of random fields see \cite{Wigman}.

\subsection{Band-limited functions.}
Another important and interesting model is given by
the so-called \emph{band-limited functions} on a smooth Riemannian manifold $\mathcal{M}$.
	Indeed, while the round sphere has spectral degeneracies, in the case of a generic manifold each eigenspace has dimension one, so that we cannot introduce a Gaussian ensemble on the eigenspace. For a generic manifold $\mathcal{M}$, let $\lambda_j$ be the eigenvalues and $f_j$ the corresponding eigenfunction. In this case, rather than considering random eigenfunctions, one considers random combinations of eigenfunctions with growing energy window of either type 

$$f^L(x)= \sum_{\lambda_j \in [0,\lambda]}a_j f_j(x) $$
(called long range dependent)
$$f^S(x)= \sum_{\lambda_j \in [\sqrt{\lambda}, \sqrt{\lambda}+1]}a_jf_j(x) $$
(called short range dependent) as $\lambda \to \infty$. In \cite{berard} and \cite{Zel} the authors found the expected value for the nodal length in the long range dependent while in \cite{Zel} the expected value in the short range. See also \cite{TotW} for the study of the number of boundary intersections for random combinations of eigenfunction $f^L(x)$ and $f^S(x)$ on generic billiards, \cite{SW19} for topologies of nodal sets
of random band-limited functions, \cite{BW} for volume distribution of their nodal domains 
	and \cite{KSW21} for nodal volume for non-Gaussian random band-limited functions.

In this paper we are interested in 
the nodal length for band-limited functions on the sphere.
We present in the next section the model and the main result of the paper.

\section{Main Results}
\subsection{Nodal length for band-limited functions on the sphere.}

In this paper we consider the following random field, 
 \begin{equation}\label{1}
\bar{T}_{\alpha_n}(x)= C_{\alpha_n}\sum_{\ell=\alpha_n n}^{n}T_\ell(x),
\end{equation}
where $T_\ell$ are defined in (\ref{ttt}), 
$\alpha_{n}$ is a sequence such that $\alpha_n=1-g(n)$ with $g(n)=o(1)$ as $n\to \infty$ and such that $ng(n) \to \infty$. The coefficient $C_{\alpha_{n}}$ is chosen such that
$$Var(\bar{T}_{\alpha_n}(x))=C_{\alpha_{n}}^2 \sum_{\ell=\alpha_n n}^{n}Var(T_\ell(x))=C_{\alpha_{n}}^2 \sum_{\ell=\alpha_n n}^{n}\frac{2\ell+1}{4\pi}=1.$$ 
Hence we take
\begin{eqnarray*}
C_{\alpha_n }^2&=& \frac{4\pi }{n^2(1-\alpha_{n}^2)+2n +1} = \frac{4\pi }{n^2(1-(1-g(n))^2)+2n +1}= \frac{4\pi }{n^2(1-1+2g(n)-g(n)^2)+2n +1}\\
&=&  \frac{4\pi }{2g(n)n^2-g(n)^2n^2+2n +1}.
\end{eqnarray*}
The covariance function of $\bar{T}_{\alpha_{n}}(x)$ is given by
\begin{equation}\label{cov}
\Gamma_{\alpha_n}(x,y):=\mathbb{E}[\bar{T}_{\alpha_n}(x)\bar{T}_{\alpha_n}(y)]= C_{\alpha_n }^2 \sum_{\ell=\alpha_n n}^n \frac{2\ell+1}{4\pi}P_{\ell}(x,y)
\end{equation}

and the nodal length of the random field $\bar{T}_{\alpha_{n}}$ is, by definition,
$$\mathcal{L}_{\alpha_{n}}:=len(\bar{T}_{\alpha_{n}}^{-1}(0)) = len\{ x \in \mathbb{S}^2: \bar{T}_{\alpha_{n}}(x)=0 \}.$$

Our main result is the following.

\begin{theorem}\label{mainth}
As $n\to \infty$, we have that
	\begin{equation}
	Var[\mathcal{L}_{\alpha_n}]=\frac{1}{32}\log n +O(1)+O(g(n)\log n).
	\end{equation}
\end{theorem}

\subsection{No phase transition and Berry's cancellation} 

Let us compare Theorem \ref{mainth} with the results obtained for one eigenfunction $T_\ell$ (which correspond to the case $\alpha_{n}\equiv 1$ in (\ref{1})).
The nodal length of $T_\ell$ has already been investigated in \cite{W, MRW2017}. More precisely, denoting $$\mathcal{L}_\ell:= len \{ x \in \mathbb{S}^2: T_\ell(x)=0 \},$$ it has been proved in \cite{W} that 
\begin{equation}\label{length1}
Var(\mathcal{L}_\ell)= \frac{\log \ell}{32}+O(1),
\end{equation}
as $\ell \to \infty$.
This result shows that the asymptotic variance is logarithmic and hence of order lower than $\ell$ which would be expected by the natural scaling considerations of the problem, as it happens for nonzero levels (see for example \cite{rossiphd}). 
Indeed, originally it was conjectured that
\begin{equation}\label{boundary}
Var(len(T_\ell^{-1}(z)))=c_1 \cdot \ell+ c_2 \cdot \log\ell+O(1)
\end{equation}
for any level $z \in \mathbb{R}$ with $c_1, c_2 \geq 0$, for $z \in \mathbb{R}$. 
However, it turned out that, when $z=0$, $c_1$ vanishes and one obtains (\ref{length1}), as predicted by Berry \cite{berry1} for the RWM. 
This phenomenon is called Berry's cancellation
phenomenon, it is due to the fact that the leading nonconstant term in the long range asymptotics of the 2-point correlation function (see \cite{W} and Section \ref{sec2point}) is purely oscillating,
so that it does not contribute to the variance and the non-oscillating leading terms cancel. 
This phenomenon seems to be of general nature. It also occurs in the torus \cite{KMW} and it is believed to hold for random combinations of eigenfunctions on a generic manifold. As we can see from Theorem \ref{mainth} it also holds in our context. 
However, when we consider an averaged of eigenfunctions on a
number of eigenfunctions that grows linearly, it results that the variance is linear
with respect to the frequency. In view of this, it is natural to conjecture that
there exists a smooth transition with respect to the number of
eigenfunctions that are averaged. This is the reason why we investigated the model (\ref{1}). 
 Althought we expected that the behaviour of the variance here could vary depending on the choice 
 of $\alpha_{n}$, Theorem \ref{mainth} shows that the asymptotic variance is still logarithmic whenever the frequency window includes a number of eigenfunctions that grows sublinearly, so that no smooth phase transition with respect to the sequence $\alpha_{n}$ arises.

\subsection{Interpretation in terms of Wiener-It\^o chaoses}\label{Chaoses}

The result in Theorem \ref{mainth} can be interpreted 
in terms of the $L^2(\Omega)$ expansion of the nodal length into Wiener-It\^o chaoses (see also the discussion in \cite{Wigman}),
 which are orthogonal spaces spanned by Hermite polynomials. First of all, we recall that the Hermite polynomials $H_q(x)$ are defined by $H_0(x)=1$, and for $q=2,3,\dots$ 
\begin{equation*}
H_q(x)=(-1)^q \frac{1}{\phi(x)} \frac{d^q\phi(x)}{dx^q},
\end{equation*}
with $ \phi(x)=\frac{1}{\sqrt{2\pi}}e^{-x^2/2}$. 

We consider the Wiener-It\^o chaos expansion 
\begin{align}\label{Proj}
\mathcal{L}_{\alpha_{n}}=\sum_{q=0}^{\infty} \mathcal{L}_{\alpha_{n}}[q],
\end{align}
where $\mathcal{L}_{\alpha_{n}}[q]$ denotes the projection of $\mathcal{L}_{\alpha_{n}}$ on the $q$-order chaos component that is the space generated by the $L^2$-completion of linear combinations of the form 
$$H_{q_1}(\xi_1) \cdot H_{q_2}(\xi_2) \cdots H_{q_k}(\xi_k), \hspace{1cm} k \ge 1,$$ 
with $q_i \in \mathbb{N}$ such that $q_1+\cdots+q_k=q$, and $(\xi_1, \dots, \xi_k)$ standard real Gaussian vector. 

The result in Theorem \ref{mainth} can be compared to those obtained for the full sphere $\mathbb S^2$ in  \cite{W, MRW2017} for the nodal length of the spherical harmonics $T_\ell(x)$. In that case, exploiting the Wiener-It\^o chaos expansion for these nonlinear functionals of Gaussian fields, 
it results that (after centering) a single term dominates the $L^2(\Omega)$ expansion, that is, the fourth chaotic component. More precisely, defining the random variable 
\begin{equation*}
h_{\ell,4}:= \int_{\mathbb{S}^2} H_4(f_\ell(x))\,dx,
\end{equation*}
called sample trispectrum, see i.e.  \cite{MRW2017, MW2014}, we have that 
$$\text{Var}(h_{\ell,4})=\frac{\log \ell}{32}+O(1).$$

Our result in Theorem \ref{mainth} suggests that the same happens for the nodal length $\bar{T}_{\alpha_{n}}$. Namely that the the random variable
$$h_{\alpha_{n}n;4}:=\int_{\mathbb{S}^2} H_4(\bar{T}_{\alpha_{n}}(x)) \,dx$$

dominates the Wiener-It\^o chaos expansion in (\ref{Proj}).\\
Note that in the case of $T_\ell$ the second chaotic projection is exactly zero while for $\bar{T}_{\alpha_{n}}$ it does not vanish. However it is not the leading term of the Wiener-Ito chaos expansion.

As a further result, in this paper, we give a direct computation of the second chaotic component of $\mathcal{L}_{\alpha_{n}}$ which can be read in the following proposition.

\begin{proposition}\label{2chaos} 
	As $n \to \infty$, we have that
	\begin{equation} 
	\begin{split}
	Var(\mathcal{L}_{\alpha_n}[2])&
	=\frac{2\pi^2}{3} g(n) \bigg[1+O\left(g(n)\right)+O\left( \frac{1}{n g(n)}  \right)\bigg].
	\end{split}
	\end{equation}
\end{proposition}

\subsection{On the proof of Theorem \ref{mainth}}

In the proof of Theorem \ref{mainth} 
we follow the same idea as in \cite{W} for the nodal length of spherical eigenfunction. To be precisely, 
we employ the Kac-Rice formula, which reduces the computation of the length variance to the 2-point correlation function, given in terms of distribution of the values $\bar{T}_{\alpha_{n}}$ as well as their gradients $\nabla \bar{T}_{\alpha_{n}} \in T_x(\mathbb{S}^2)$, for all $x \in \mathbb{S}^2$. 
We derive then the 2-point correlation function for the field defined in (\ref{1}) and we investigate its asymptotic behaviour. To do so we need to study the covariance function (\ref{cov}) and its first and second derivatives in the high energy limit. 
Note that usually the analysis of the variance of $T_\ell$ exploits the Hilb's asymptotic (formula (8.21.17) on page 197 \cite{szego}):

\begin{equation}\label{hilbj0}
 P_\ell (\cos \theta)= \left( \frac{\theta}{\sin \theta}\right)^{1/2} J_0((\ell+1/2)\theta)+\delta(\theta),
\end{equation}
uniformly for $0 \leq \theta \leq \pi/2, $ where $J_0$ is the Bessel function of order 0 and the error term is 
\begin{equation*}
\delta(\theta) \ll \begin{cases}
\theta^{1/2}O(\ell^{-3/2}), & C\ell^{-1}<\theta <\pi/2\\
\theta^2 O(1), & 0<\theta <C\ell^{-1},
\end{cases}
\end{equation*}
where $ C>0$ is any constant and the constants involved in the "$O"-$notation depend on $C$ only.

Here we cannot directly use (\ref{hilbj0}) inside (\ref{cov}). Indeed substituting the Hilb's asymptotic inside the sum in (\ref{cov}) we are not able to control the error term $\sum_{\ell=\alpha_nn}^{n} \delta(\theta)$, which grows faster than the leading term. 
To avoid this problem we exploit the following formula, found in \cite{FXA} and derived by the
so-called Cristoffel-Darboux formula \cite{AH12},

	\begin{equation}\label{CDform}
\sum_{\ell=0}^{n} \sum_{m=-\ell}^{\ell} Y_{\ell, m} (x) Y_{\ell, m} (y)= \frac{n+1}{4\pi} P_n^{(1,0)}(\cos \theta(x,y)),
\end{equation} 

where $P_n^{(1,0)}$ is a Jacobi Polynomial, defined in general as
\begin{equation*}
P_n^{(\alpha, \beta)}(x)= \sum_{s=0}^n \binom{n+\alpha}{s} \binom{n+\beta}{%
	n-s} \bigg(\frac{x-1}{2}\bigg)^{n-s}\bigg(\frac{x+1}{2}\bigg)^s,
\end{equation*}
\begin{equation*}
R_{1,n}(\theta) =%
\begin{cases}
\theta^3 O(n), & 0\leq \theta \leq c/n \\ 
\theta^{1/2} O(n^{-3/2}), & c/n \leq \theta \leq \pi -\epsilon%
\end{cases}%
\end{equation*}
and $J_1$ is the Bessel function of order 1. 
Note that in our notation we have $P_\ell(x)=P_{\ell}^{(1,0)}(x).$

(\ref{CDform}) allows to write the covariance function as

	$$\Gamma_{\alpha_{n}}(\cos \theta)=C_{\alpha_{n}}^2 \bigg[ \frac{n+1}{4\pi} P_n^{(1,0)}(\cos \theta(x,y))-\frac{n\alpha_{n}}{4\pi} P_{n\alpha_{n}-1}^{(1,0)} (\cos \theta(x,y)) \bigg]$$

and hence to get rid of the sum. At this point we can appeal to the asymptotic for Jacobi polynomials, given in theorem 8.21.13 in \cite{szego}, namely

$$P_n^{\alpha,\beta}(\cos \theta)=n^{-1/2} k(\theta) \cos(N\theta+\gamma)+O(n^{-1/2}),$$
where 

$$k(\theta)=\pi^{-1/2} (\sin \theta/2)^{-\alpha-1/2}(\cos \theta/2)^{-\beta-1/2},$$

$$N=n+(\alpha+\beta+1)/2 \quad \mbox{ and } \quad \gamma=-(\alpha+1/2)\pi/2.$$

 To derive also the asymptotic behaviour of the first and second derivatives of the covariance function we exploit that 
$$\frac{d}{dx}P_n^{a,b}(x)=\frac{1}{2}(a+b+n+1) P_{n-1}^{(a+1,b+1)}(x)$$
(see (4.5.5) \cite{szego}) and proceed similarly.

\subsection{Plan of the paper}
The paper is organized as follows. In Section \ref{integrale section} we start giving the first and second moment of $\mathcal{L}_{\alpha_{n}}$. In particular we show an explicit formulation of the second moment in terms of the two-point correlation function and we give its asymptotic expansion in the high energy limit (Lemma \ref{asympK}). In Section \ref{asympvarsection} we prove Theorem \ref{mainth} exploiting Lemma \ref{asympK}, which has been proved in Section \ref{twopoint}.
 In Section \ref{secondprojsection} we compute the second chaotic component of $\mathcal{L}_{\alpha_{n}}$, proving Proposition \ref{2chaos}. Finally in Appendix \ref{asympcovsection} the asymptotic behaviour of the covariance function and its derivatives have been investigated, while Appendix \ref{technicalsection} collects some technical results used in the proof of Lemma \ref{asympK}.
 
\subsection*{Acknowledgements}
The author would like to thank Domenico Marinucci and Igor Wigman for suggesting this problem, for the helpful discussions and comments. The author is also grateful to Maurizia Rossi for some insightful conversations. 
This work has been partially supported 
by "Progetto di Ricerca GNAMPA-INdAM", 
codice CUP\_E55F22000270001.

\section{An Explicit Integral Formula for the second moment}\label{integrale section}

In this section we give the first moment of the nodal length $\mathcal{L}_{\alpha_{n}}$ and we express its second moment via the Kac-Rice formula introducing the two-point correlation function. 

First of all let us consider the gradient $$\nabla \bar{T}_{\alpha_n}(x)= C_{\alpha_n} \sum_{\ell=\alpha_nn}^n \nabla T_\ell(x);$$ the nodal length can be formally written 
as  $$\mathcal{L}_{\alpha_n}= \int_{\mathbb{S}^2} \delta_0 (\bar{T}_{\alpha_n}(x)) ||\nabla {\bar{T}_{\alpha_n}}(x)||\,dx$$
(as seen for example in \cite{MRW2017}). 

The variance of each component of $\nabla \bar{T}_{\alpha_n}(x)$ is given by $$D_{\alpha_{n}}:=C_{\alpha_n}^2 \sum_{\ell=\alpha_nn}^n \dfrac{2\ell+1}{4\pi}\dfrac{\ell(\ell+1)}{2}.$$

We normalize the gradient such that the variance of the components is equal to 1.
Hence

$$\mathcal{L}_{\alpha_n}= \sqrt{ D_{\alpha_{n}}} \int_{\mathbb{S}^2} \delta_0 (\bar{T}_{\alpha_n}(x)) ||\tilde{\nabla}{\bar{T}}_{\alpha_n}(x)||\,dx, \quad \mbox{ with } \quad \tilde{\nabla}{\bar{T}}_{\alpha_n}(x):= \frac{\nabla \bar{T}_{\alpha_n}(x)}{\sqrt{D_{\alpha_{n}}}}.$$

\subsection{First moment of the nodal length}
For $\epsilon>0$ we define

$$\mathcal{L}_{\alpha_n}^\epsilon:=\int_{\mathbb{S}^2} ||\nabla \bar{T}_{\alpha_{n}}(x)|| \chi_\epsilon(\bar{T}_{\alpha_{n}}(x))\, dx.$$

The same steps in \cite{MRW2017} and in \cite{ST} prove that, as $\epsilon \to 0,$ $$\lim_{\epsilon \to 0} \mathbb{E}[|\mathcal{L}_{\alpha_n}^\epsilon-\mathcal{L}_{\alpha_n}|^2]=0$$ in the $L^2-$sense. 
Then the mean of the nodal length is given by 
\begin{eqnarray*}
\mathbb{E}[\mathcal{L}_{\alpha_nn}]&=& \lim_{\epsilon \to 0} \mathbb{E}[\mathcal{L}_{\alpha_nn}^\epsilon]= \lim_{\epsilon \to 0} \dfrac{1}{2\epsilon} \sqrt{D_{\alpha_{n}}} \mathbb{E}[1_{-\epsilon,\epsilon}(\bar{T}_{\alpha_n}(x)) ||\tilde{\nabla}\bar{T}_{\alpha_n}(x)||]\\
&=& \sqrt{D_{\alpha_{n}}} \dfrac{|\mathbb{S}^2|}{2}.
\end{eqnarray*}

\subsection{Two-point correlation function and second moment} \label{sec2point}
In this section we follow the steps in \cite{W} readapted to our field.
One may define the Two-point correlation function $\tilde{K}(x,y)$ intrinsically as

$$\tilde{K}(x,y)= \frac{1}{(2\pi)\sqrt{1-\Gamma_{\alpha_n}(x,y)^2}} \mathbb{E}[||\nabla \bar{T}_{\alpha_{n}}(x)||\cdot ||\nabla \bar{T}_{\alpha_{n}}(y)|| | \bar{T}_{\alpha_{n}}(x) =\bar{T}_{\alpha_{n}}(x) =0].$$

Now we present the Kac-Rice formula. Given $x,y \in \mathbb{S}^2$ we consider the two local orthonormal frames $F^x(z)=\{ e_1^x, e_2^x \}$ and $F^y(z)=\{ e_1^y, e_2^y\}$ , defined in some neighbourhood of $x$ and $y$ respectively. 

For $x, y \in \mathbb{S}^2$
we define the following random vectors:
\begin{equation}\label{Z}
Z= (\bar{T}_{\alpha_n}(x), \bar{T}_{\alpha_n}(y), \tilde{\nabla}{\bar{T}_{\alpha_n}}(x), \tilde{\nabla}{\bar{T}_{\alpha_n}}(y)) \in \mathbb{R}^2 \times T_x(\mathbb{S}^2)\times T_y(\mathbb{S}^2).
\end{equation}

The random vector $Z$ is a mean zero Gaussian with covariance matrix 

$$\Sigma= \Sigma(x,y)=\begin{pmatrix}
A & B\\ B^t &C
\end{pmatrix},$$

where
\begin{equation}\label{matrixA}
A= A_{{\alpha_n}}(x,y)=\begin{pmatrix} 1 & \Gamma_{\alpha_n} \\ \Gamma_{\alpha_n} & 1
\end{pmatrix},
\end{equation}

\begin{equation}\label{matrixB}
B=B_{{\alpha_n}}(x,y) =\begin{pmatrix} 0 & \nabla_y \Gamma_{\alpha_n} \\ \nabla_x \Gamma_{\alpha_n}& 0
\end{pmatrix}
\end{equation}

and \begin{equation}\label{matrixC}
C=C_{{\alpha_n}}(x,y)= \begin{pmatrix} \frac{C_{\alpha_n}^2}{2} \sum_{\ell=\alpha_n n}^{n} \ell(\ell+1)I_2 & H \\ H^t&   \frac{C_{\alpha_n}^2}{2} \sum_{\ell=\alpha_n n}^{n} \ell(\ell+1)I_2
\end{pmatrix},
\end{equation}
with $H$ the "pseudo-Hessian" 
$$H_{\alpha_n}(x,y)=(\nabla_x \otimes \nabla_y) \Gamma_{\alpha_n}(x,y),$$ i.e. 

$H=(h_{jk})_{j,k=1,2}$
 with entries given by $h_{jk}= \frac{\partial^2}{\partial e_j^x \partial e_k^y} \Gamma_{\alpha_{n}}(x,y)$.

The covariance matrix of the Gaussian distribution of $Z$  in (\ref{Z}) conditioned upon $\bar{T}_{\alpha_n}(x)=\bar{T}_{\alpha_{n}}(y)=0$ is given by $$\Omega_{\alpha_{n}}(x,y)= C-B^tA^{-1}B.$$ Then the two-point correlation function is 

\begin{equation}
\begin{split}
\tilde{K}_{\alpha_{n}}(x,y)=\dfrac{1}{\sqrt{1-\Gamma_{\alpha_{n}}(x,y)^2}} \int_{\mathbb{R}^2\times \mathbb{R}^2 } ||W_1|| ||W_2||e^{-\frac{1}{2}(w_1,w_2) \Omega_{\alpha_{n}}^{-1}(w_1,w_2)^T} \dfrac{dw_1dw_2}{(2\pi)^3 \sqrt{det \Omega_{\alpha_{n}}(x,y)}}.
\end{split}
\end{equation}

Let $\theta, \varphi$ be the standard spherical coordinates on $\mathbb{S}^2$. Using the rotational invariance of the 2-point correlation function (see Remark 2.4 \cite{W}) we obtain
\begin{equation}
\mathbb{E}[\mathcal{L}_{\alpha_n}^2]=\int_{\mathbb{S}^2\times \mathbb{S}^2} \tilde{K}(x,y) \,dxdy= 2\pi |\mathbb{S}^2| \int_{0}^{\pi} \tilde{K}_{\alpha_{n}}(N,x(\theta)) \sin \theta \,d\theta
\end{equation}

where $x(\theta) \in \mathbb{S}^2$ is the point corresponding to the spherical coordinates $( \theta,\varphi=0)$. Note that $\tilde{K}(N,x(\theta))=\tilde{K}(x,y)$ for any $x,y \in \mathbb{S}^2$ with $d(x,y)=\theta.$ We have that

\begin{equation}\label{momento2}
\mathbb{E}[\mathcal{L}_{\alpha_n}^2]= 2\pi |\mathbb{S}^2| \int_{0}^{\pi} \tilde{K}_{\alpha_{n}}(\theta) \sin \theta \,d\theta,
\end{equation}
where $\tilde{K}_{\alpha_{n}}(\theta)=\tilde{K}_{\alpha_{n}}(x,y)$, $x,y \in \mathbb{S}^2$ being any pair of points with $d(x,y)=\theta$ (see Corollary 2.5 \cite{W}).

We recall that from the fact that $\tilde{K}(\psi)=\tilde{K}(\pi-\psi)$ we can integrate in the hemisphere, see \cite{W}, and then from (\ref{momento2}), rescaling $\theta=\psi/{(\alpha_n m)}$, we note that
the variance of the nodal length can be written as
\begin{equation}
Var[\mathcal{L}_{\alpha_n}]=16\pi^2 \frac{D_{\alpha_{n}}}{m\alpha_{n}}\int_{0}^{ m\alpha_{n}\pi/2} \left(K_{\alpha_{n}}(\psi)-\frac{1}{4}\right) \sin \frac{\psi}{m\alpha_{n}}\,d\psi,
\end{equation}
with $m=n+1/2$ (see \cite{W} for more details).

The main goal of the present paper reduces to understand the asymptotic behaviour of the
function ${K}_{\alpha_{n}}(\psi)$, which is given in the following lemma and whose proof can be found in Section \ref{twopoint}.

\begin{lemma}\label{asympK}
	For $C<\psi<(\pi/2) m\alpha_{n}$, with  $C >0$, denoting 
	\[ h:= \sum_{ k=1}^{\infty} g(n)^k+ \frac{1}{n}  +\frac{1}{2n} \sum_{k=1}^{\infty} g(n)^k, \] we have that, as $n\to \infty$,
	\begin{eqnarray}\label{Kbigtheta}
	K(\psi) &=& \frac{1}{4} + \frac{1}{256\pi^2 \psi^2}
+  \frac{1}{2\pi\psi} \sin\left(h\psi+2\psi +\frac{\psi}{2n} +O\left(\frac{\psi}{n^2}\right)\right)-\frac{75}{256\pi^2\psi^2} \cos\left(2h\psi+4\psi +\frac{2\psi}{n} +O\left(\frac{\psi}{n^2}\right)\right) \nonumber\\&&	
	+\frac{27}{64\pi^2\psi^2}  \sin\left(h\psi+2\psi-\frac{\psi}{n}+O\left(\frac{\psi}{n^2}\right)\right)
	-\frac{1}{4\pi\psi^2}\cos\left(h\psi+2\psi-\frac{\psi}{n}+O\left(\frac{\psi}{n^2}\right)\right) \nonumber\\&&+\frac{1}{4\pi \psi^2}\sin\left(\frac{h\psi}{2}+\psi-\frac{\psi}{2n}+O\left(\frac{\psi}{n^2}\right)\right)\cos\left(\psi-\frac{\psi}{2n}+O\left(\frac{\psi}{n^2}\right)\right)\nonumber\\&&
	-\frac{3}{2\pi \psi^2 } \sin\left(\frac{h\psi}{2}+\psi-\frac{\psi}{2n}+O\left(\frac{\psi}{n^2}\right)-\frac{\pi}{4}\right)\cos\left(\psi-\frac{\psi}{2n}+O\left(\frac{\psi}{n^2}\right)-\frac{5}{4}\pi\right)\nonumber\\&&+O\left(\frac{1}{\psi^3}+\frac{g(n)}{\psi^2}+\frac{1}{g(n)n\psi^2}\right)  .\nonumber
	\end{eqnarray}
\end{lemma}

\section{Asymptotic for the variance: Proof of Theorem \ref{mainth}}\label{asympvarsection}

In this section we prove Theorem \ref{mainth}. Let us define 

$$I_1:=16\pi^2 \frac{D_{\alpha_n}}{m\alpha_{n}}\int_{0}^{C} \left(K_{\alpha_{n}}(\psi)-\frac{1}{4}\right) \sin \frac{\psi}{m\alpha_{n}}\,d\psi$$
and
$$I_2:=16\pi^2 \frac{D_{\alpha_n}}{m\alpha_{n}}\int_{C}^{ m\alpha_{n}\pi/2} \left(K_{\alpha_{n}}(\psi)-\frac{1}{4}\right) \sin \frac{\psi}{m\alpha_{n}}\,d\psi.$$

We have the following two propositions. 
\begin{proposition}\label{big theta} For any constant $C>0$ and
	$C<\psi < m\alpha_{n}\pi/2$, as $n\to \infty,$
	
	$$I_2 = \frac{1}{32}\log n +O(1)+O(g(n)\log n).$$
	
\end{proposition}

	\begin{proof}
Exploiting the expansion in Lemma (\ref{asympK}) we get
	\begin{eqnarray*}
I_2&=& 16\pi^2 n^2\frac{1}{2(m\alpha_n)^2} \int_{C}^{ m \alpha_n\pi/2} \bigg(\frac{1}{256\pi^2 \psi^2}
	+  \frac{1}{2\pi\psi} \sin\left(h\psi+2\psi +\frac{\psi}{2n} +O\left(\frac{\psi}{n^2}\right)\right)\\&&-\frac{75}{256\pi^2\psi^2} \cos\left(2h\psi+4\psi +\frac{2\psi}{n} +O\left(\frac{\psi}{n^2}\right)\right) \nonumber\\&&	
	+\frac{27}{64\pi^2\psi^2}  \sin\left(h\psi+2\psi-\frac{\psi}{n}+O\left(\frac{\psi}{n^2}\right)\right)
	-\frac{1}{4\pi\psi^2}\cos\left(h\psi+2\psi-\frac{\psi}{n}+O\left(\frac{\psi}{n^2}\right)\right) \nonumber\\&&+\frac{1}{4\pi \psi^2}\sin\left(\frac{h\psi}{2}+\psi-\frac{\psi}{2n}+O\left(\frac{\psi}{n^2}\right)\right)\cos\left(\psi-\frac{\psi}{2n}+O\left(\frac{\psi}{n^2}\right)\right)\nonumber\\&&
	-\frac{3}{2\pi \psi^2 } \sin\left(\frac{h\psi}{2}+\psi-\frac{\psi}{2n}+O\left(\frac{\psi}{n^2}\right)-\frac{\pi}{4}\right)\cos\left(\psi-\frac{\psi}{2n}+O\left(\frac{\psi}{n^2}\right)-\frac{5}{4}\pi\right)\nonumber\\&&+O\left(\frac{1}{\psi^3}+\frac{g(n)}{\psi^2}+\frac{1}{g(n)n\psi^2}\right) \bigg)\psi\,d\psi .\nonumber
\end{eqnarray*}
Solving the integrals and noting that $\int_{C}^{\alpha_{n}n} \frac{\sin x}{x} \,dx=O(1)$ and $\int_{C}^{\alpha_{n}n} \sin x \,dx=O(1)$ we find the thesis of the proposition.
\end{proof}

\begin{proposition}\label{small theta}
For	any constant $C>0$ we have, as $n\to \infty$

	$$\int_{0}^{C} \bigg|K_{\alpha_{n}}(\psi)-\frac{1}{4}\bigg| \sin\left(\frac{\psi}{\alpha_{n}m}\right) \, d\psi = O\left(\frac{1}{n}\right). $$
\end{proposition}

The proof of Proposition \ref{small theta} is collected in Appendix \ref{sectionsmalltheta}.

Proposition \ref{small theta} and Proposition \ref{big theta} complete the proof of Theorem \ref{mainth}. Indeed, as $n \to \infty$, we get

\begin{eqnarray*}
Var[\mathcal{L}_{\alpha_n}]&=&16\pi^2 \frac{D_{\alpha_n}}{m\alpha_{n}}\int_{0}^{ m\alpha_{n}\pi/2} \left(K_{\alpha_n}(\psi)-\frac{1}{4}\right) \sin \frac{\psi}{m\alpha_{n}}\,d\psi=I_1+I_2\\
&=&\frac{1}{32}\log n+O(g(n)\log n)+O(1)
\end{eqnarray*}
	
\section{The two-point correlation function: proof of Lemma \ref{asympK}} \label{twopoint}
We write here an explicit expression of the two-point correlation function and we prove Lemma \ref{asympK}. Denoting by $N$ the North Pole, we fix $x=N$ and, in view of the isotropy, using the spherical coordinates on $\mathbb{S}^2$ we write  ${\Gamma}_{\alpha_{n}}(x,y)={\Gamma}_{\alpha_{n}}(\cos
\theta),$ with $\theta \in [0,\pi)$, and then
the determinant of A, defined in (\ref{matrixA}) is
$$det A=1-\Gamma_{\alpha_{n}}(\cos\theta)^2= 1- \left(C_{\alpha_n}^2\sum_{\ell=\alpha_n n}^{n} \frac{2\ell+1}{4\pi} P_\ell(\cos \theta)\right)^2.$$

The matrix B in (\ref{matrixB}) can be written as

\begin{equation}
B= \pm C_{\alpha_n}^2 \begin{pmatrix} 0 & 0&  \sum_{\ell=\alpha_n n}^{n} \frac{2\ell+1}{4\pi}P_\ell^\prime(\cos \theta) \sin \theta & 0 \\ -\sum_{\ell=\alpha_n n}^{n}\frac{2\ell+1}{4\pi} P_\ell^\prime(\cos \theta) \sin \theta  & 0 & 0 & 0
\end{pmatrix}
\end{equation}

and 

\begin{equation}
B^tA^{-1}B= \dfrac{1}{1-  \Gamma_{\alpha_n}(\cos\theta)^2} \times  
\begin{pmatrix} \Gamma_{\alpha_{n}}'(\cos \theta)^2 &0 & \Gamma_{\alpha_{n}}'(\cos \theta)^2 \Gamma_{\alpha_{n}}(\cos \theta) &0 \\
0 &0&0&0\\
\Gamma_{\alpha_{n}}'(\cos \theta)^2 \Gamma_{\alpha_{n}}(\cos \theta)&0& \Gamma_{\alpha_{n}}'(\cos \theta)^2 &0\\
0&0&0&0
\end{pmatrix},
\end{equation}
where we recall that
$\Gamma_{\alpha_{n}}(\cos \theta)=  C_{\alpha_n}^2\sum_{\alpha_n n}^{n} \frac{2\ell+1}{4\pi}P_\ell(\cos \theta)$ is defined in (\ref{cov}) and then its derivative is
$$\Gamma_{\alpha_{n}}'(\cos \theta) =- C_{\alpha_n}^2\sum_{\ell=\alpha_n n}^{n} \frac{2\ell+1}{4\pi}P_\ell^\prime(\cos \theta) \sin \theta.$$

Moreover,

\begin{equation}
H= C_{\alpha_n}^2\begin{pmatrix} \sum_{\ell=\alpha_n n}^{n} \frac{2\ell+1}{4\pi} P_\ell^\prime(\cos \theta) \cos \theta -\frac{2\ell+1}{4\pi} P_\ell^{\prime \prime} (\cos \theta )\sin^2(\theta)& 0 \\ 0& \sum_{\ell=\alpha_n n}^{n} \frac{2\ell+1}{4\pi} P_\ell^\prime(\cos \theta)
\end{pmatrix}
\end{equation}
and 
$$\Omega_{\alpha_{n}}(x,y)= C-B^tA^{-1}B$$

\begin{equation}
=\begin{pmatrix} 
\frac{C_{\alpha_n}^2}{2} \sum_{\ell=\alpha_n n}^{n} \ell(\ell+1)+ \tilde{a}& 0&  \tilde{b}&0\\
0 & \frac{C_{\alpha_n}^2}{2} \sum_{\ell=\alpha_n n}^{n} \ell(\ell+1)& 0& \tilde{c}\\
\tilde{b}&0 & \frac{C_{\alpha_n}^2}{2} \sum_{\ell=\alpha_n n}^{n} \ell(\ell+1)+ \tilde{a} &0\\ 
0& \tilde{c} &0& \frac{C_{\alpha_n}^2}{2} \sum_{\ell=\alpha_n n}^{n} \ell(\ell+1)\\
\end{pmatrix} ,
\end{equation}

where
\begin{eqnarray*}
\tilde{a}&=&\tilde{a}_{\alpha_{n}}(\theta)=-\frac{1}{1 -  \left(C_{\alpha_n}^2\sum_{\alpha_n n}^{n} \frac{2\ell+1}{4\pi}  P_\ell(\cos \theta)\right)^2}  \left(C_{\alpha_n}^2\sum_{\ell=\alpha_n n}^{n} \frac{2\ell+1}{4\pi} P_\ell^\prime(\cos \theta) \sin \theta\right)^2,\\
\tilde{b}&=&\tilde{b}_{\alpha_n}(\theta)=C_{\alpha_n}^2 \left(\sum_{\ell=\alpha_n n}^{n} \frac{2\ell+1}{4\pi} P_\ell^\prime(\cos \theta) \cos \theta -\frac{2\ell+1}{4\pi} P_\ell^{\prime \prime} (\cos \theta )\sin^2(\theta)\right)\\
&&\qquad-  \dfrac{\left(C_{\alpha_n}^2\sum_{\alpha_n n}^{n} \frac{2\ell+1}{4\pi} P_\ell^\prime(\cos \theta)\sin \theta \right)^2 }{1-  \left(C_{\alpha_n}^2\sum_{\alpha_n n}^{n} \frac{2\ell+1}{4\pi} P_\ell(\cos \theta)\right)^2 }\left( C_{\alpha_n}^2\sum_{\ell=\alpha_n n}^{n}\frac{2\ell+1}{4\pi}  P_\ell(\cos \theta) \right),\\
\tilde{c}&=&\tilde{c}_{\alpha_{n}}(\theta)= C_{\alpha_n}^2\sum_{\ell=\alpha_n n}^{n}\frac{2\ell+1}{4\pi}  P_\ell^\prime (\cos \theta) .
\end{eqnarray*}

 Now rescaling $\theta= \frac{\psi}{\alpha_{n} m}$, we define

\begin{equation}
\Delta_{\alpha_{n}}(\psi):= \dfrac{\Omega_{\alpha_{n}}(\psi/m\alpha_{n})}{\frac{C_{\alpha_n}^2}{2} \sum_{\ell=\alpha_n n}^{n} \ell(\ell+1)}=\begin{pmatrix} 
1+ 2a& 0&  2b&0\\
0 & 1& 0& 2c\\
2b&0 & 1+2a &0\\ 
0& 2c &0& 1\\
\end{pmatrix} ,
\end{equation}

where 

\begin{eqnarray*}
a&:=&a_{\alpha_{n}}(\psi)=- \dfrac{1}{C_{\alpha_n}^2\sum_{\ell=\alpha_n n}^{n} \ell(\ell+1)}\tilde{a}_{\alpha_{n}}(\psi/({\alpha_{n}}m)),\\
b&:=&b_{\alpha_{n}}(\psi)= \dfrac{1}{C_{\alpha_n}^2\sum_{\ell=\alpha_n n}^{n} \ell(\ell+1)}\tilde{b}_{\alpha_{n}}(\psi/({\alpha_{n}}m)),\\
c&:=&c_{\alpha_{n}}(\psi)= \dfrac{1}{C_{\alpha_n}^2\sum_{\ell=\alpha_n n}^{n} \ell(\ell+1)}\tilde{c}_{\alpha_{n}}(\psi/({\alpha_{n}}m)).
\end{eqnarray*}

From Remark 2.10 in \cite{W}, we know that
\begin{equation}
K_{\alpha_{n}}(\psi)=
\dfrac{1}{2\pi \sqrt{1- \left(\Gamma_{\alpha_{n}}(\cos(\psi/(\alpha_{n}m)))\right)^2}} \mathbb{E}[||U|| ||V||],
\end{equation}
where
$(U,V)$ is a mean zero Gaussian random vector with covariance matrix $\Delta_{\alpha_{n}}(\psi)$ and 

$$\mathbb{E}[||U|| ||V||]= \dfrac{\pi}{2}+\dfrac{\pi}{2} a+\dfrac{\pi}{4} b^2-\dfrac{\pi}{16}a^2-\dfrac{3\pi}{8}ab^2+\dfrac{3\pi}{64}b^4+O(a^3+b^5+c^2).$$

Now we also have the expansion
$$\dfrac{1}{ \sqrt{1- \left(\Gamma_{\alpha_{n}}\left(\cos\frac{\psi}{\alpha_{n}m}\right)\right)^2}} 
 =1+ \frac{1}{2} \left(\Gamma_{\alpha_{n}}\left(\cos\frac{\psi}{\alpha_{n}m}\right)\right)^2+ \frac{3}{8} \left(\Gamma_{\alpha_{n}}\left(\cos\frac{\psi}{\alpha_{n}m}\right)\right)^4
+O \left( \left(\Gamma_{\alpha_{n}}\left(\cos\frac{\psi}{\alpha_{n}m}\right)\right)^6\right)$$

and hence

\begin{equation}\label{K}
\begin{split}
K_{\alpha_{n}}(\psi)&= \dfrac{1}{2\pi} \bigg[ 1+ \dfrac{1}{2} \left( \Gamma_{\alpha_{n}}\left(\cos\frac{\psi}{\alpha_{n}m}\right)\right)^2+ \dfrac{3}{8}  \left( \Gamma_{\alpha_{n}}\left(\cos\frac{\psi}{\alpha_{n}m}\right)\right)^4 +O \left(\left(\Gamma_{\alpha_{n}}\left(\cos\frac{\psi}{\alpha_{n}m}\right)\right)^6\right)\\
& \quad \times \dfrac{\pi}{2} \bigg[ 1+a+\dfrac{b^2}{2}-\dfrac{a^2}{8}-\dfrac{3}{4}ab^2+\dfrac{3}{32}b^4+O(a^3+b^5+c^2) \bigg]\\&
=  \dfrac{1}{2\pi} \frac{\pi}{2}\bigg[ 1+a+\frac{b^2}{2} -\frac{a^2}{8} -\frac{3}{4} ab^2 +\frac{3}{32} b^4 +\frac{1}{2} \left( \Gamma_{\alpha_{n}}\left(\cos\frac{\psi}{\alpha_{n}m}\right)\right)^2 
+\frac{1}{2} \left( \Gamma_{\alpha_{n}}\left(\cos\frac{\psi}{\alpha_{n}m}\right)\right)^2 a  \\&\quad+ \frac{1}{4} \left( \Gamma_{\alpha_{n}}\left(\cos\frac{\psi}{\alpha_{n}m}\right)\right)^2 b^2+ \dfrac{3}{8} \left( \Gamma_{\alpha_{n}}\left(\cos\frac{\psi}{\alpha_{n}m}\right)\right)^4 + O\left(\left( \Gamma_{\alpha_{n}}\left(\cos\frac{\psi}{\alpha_{n}m}\right)\right)^6+a^3+b^5+c^2\right) \bigg].\\
&=\frac{1}{4} \bigg\{ 1+a+\frac{b^2}{2}+\frac{1}{{2}}\Gamma_{\alpha_{n}}^2\left(\cos \frac{\psi}{\alpha_{n}m}\right)-\frac{a^2}{8} -\frac{3}{4}ab^2+\frac{3}{32}b^4+\frac{1}{{2}}\Gamma_{\alpha_{n}}^2 \left(\cos \frac{\psi}{\alpha_{n}m}\right)a \\&\quad+\frac{1}{4}\Gamma_{\alpha_{n}}^2\left(\cos \frac{\psi}{\alpha_{n}m}\right) b^2+ \frac{3}{8}\Gamma_{\alpha_{n}}^4\left(\cos \frac{\psi}{\alpha_{n}m}\right)+O\left(\Gamma_{\alpha_{n}}^6\left(\cos \frac{\psi}{\alpha_{n}m}\right)+a^3+b^5+c^2\right) \bigg\}.\\
\end{split}
\end{equation}

Now, exploiting the asymptotic behaviour of each single term, which can be found in Appendix \ref{asympcovsection} 
and \ref{technicalsection},
we obtain

\begin{eqnarray*}
	K_{\alpha_{n}}(\psi)
	&=& \frac{1}{4}+ \frac{1}{4}\bigg\{ \frac{1}{64\pi^2\psi^2}+ \frac{2}{\pi\psi} \sin\left(h\psi+2\psi +\frac{\psi}{2n} +O\left(\frac{\psi}{n^2}\right)\right)-\frac{75}{64\pi^2\psi^2} \cos\left(2h\psi+4\psi +\frac{2\psi}{n} +O\left(\frac{\psi}{n^2}\right)\right)\\&&	
	+\frac{54}{32\pi^2\psi^2}  \sin\left(h\psi+2\psi-\frac{\psi}{n}+O\left(\frac{\psi}{n^2}\right)\right)
	-\frac{1}{\pi\psi^2}\cos\left(h\psi+2\psi-\frac{\psi}{n}+O\left(\frac{\psi}{n^2}\right)\right)\\&&+\frac{1}{\pi \psi^2}\sin\left(\frac{h\psi}{2}+\psi-\frac{\psi}{2n}+O\left(\frac{\psi}{n^2}\right)\right)\cos\left(\psi-\frac{\psi}{2n}+O\left(\frac{\psi}{n^2}\right)\right)\\&&
	-\frac{6}{\pi \psi^2 } \sin\left(\frac{h\psi}{2}+\psi-\frac{\psi}{2n}+O\left(\frac{\psi}{n^2}\right)-\frac{\pi}{4}\right)\cos\left(\psi-\frac{\psi}{2n}+O\left(\frac{\psi}{n^2}\right)-\frac{5}{4}\pi\right)\\&&+O\left(\frac{1}{\psi^3}+\frac{g(n)}{\psi^2}+\frac{1}{g(n)n\psi^2}\right)  \bigg\},
\end{eqnarray*}

which leads to (\ref{Kbigtheta}).

\section{Second chaotic component: Proof of Proposition \ref{2chaos}}\label{secondprojsection}

In this section we compute explicitly the second chaotic component of the Wiener-It\^o chaos expansion mentioned in Section \ref{Chaoses}. We prove that its variance is $O(g(n))$ and then $\mathcal{L}_{\alpha_{n}}[2]$ cannot be the leading term of the chaos expansion of the nodal length $\mathcal{L}_{\alpha_{n}}$ given in (\ref{Proj}).

Before proving Proposition \ref{2chaos}, note that 
\begin{eqnarray*}
\sum_{\ell=\alpha_nn}^{n} \frac{2\ell+1}{4\pi}&=& \frac{1}{4\pi} \bigg[ n^2+n-(n\alpha_{n})^2+n\alpha_{n}+n-\alpha_{n}n+1 \bigg]\\
&=&\frac{2n^2g(n)}{4\pi} \bigg[ 1-\frac{g(n)}{2}+\frac{1}{g(n)n}+\frac{1}{2n^2g(n)} \bigg]
\end{eqnarray*}

and then 
\begin{equation}\label{C2}
C_{\alpha_n}^2=\frac{1}{\sum_{\ell=\alpha_nn}^{n} \frac{2\ell+1}{4\pi}}= \frac{4\pi}{2n^2g(n)\bigg[ 1-\frac{g(n)}{2}+\frac{1}{g(n)n}+\frac{1}{2n^2g(n)} \bigg]}
\end{equation}
which implies that
$$C_{\alpha_n}^2=\frac{4\pi}{2n^2g(n)} \left(1+\sum_{k=1}^{\infty}  \left(\frac{g(n)}{2}-\frac{1}{g(n)n}-\frac{1}{2n^2g(n)} \right)^k\right).$$

\begin{proof}[Proof of Proposition \ref{2chaos}]
The second chaotic projection is given by
	\begin{equation}
	\begin{split}
	\mathcal{L}_{\alpha_{n}}[2]= \sqrt{D_{\alpha_n}} \bigg( \frac{\beta_0 \alpha_{00}}{2} \int_{\mathbb{S}^2}  H_2(T_{\alpha_n}(x))\,dx+\frac{\beta_0 \alpha_{20}}{2} \int_{\mathbb{S}^2} ( \langle \tilde{\nabla} T_{\alpha_n}(x) , \tilde{\nabla} T_{\alpha_n}(x)  \rangle-2 )\,dx \bigg) 
	\end{split}
	\end{equation}
(see for example \cite{Rossi}),	where 
	$\alpha_{00}= \sqrt{\frac{\pi}{2}}$, $\alpha_{20}=\sqrt{\frac{\pi}{2}} \frac{1}{2},$ $\beta_0=\frac{1}{\sqrt{2\pi}}$, $\beta_2=-\frac{1}{\sqrt{2\pi}}.$
	
	Green's formula implies
	\begin{eqnarray*}
	\int_{\mathbb{S}^2} \langle \tilde{\nabla} T_{\alpha_n}(x) , \tilde{\nabla} T_{\alpha_n}(x)  \rangle \,dx&=&\frac{1}{D_{\alpha_n}} \int_{\mathbb{S}^2} \langle \nabla T_{\alpha_n}(x) , \nabla T_{\alpha_n}(x)  \rangle dx=\frac{1}{D_{\alpha_n}} \bigg(-\int_{\mathbb{S}^2}  T_{\alpha_n}(x)  \Delta T_{\alpha_n}(x)  dx \bigg)\\
	&=&\frac{1}{D_{\alpha_n}} \bigg(-\int_{\mathbb{S}^2} (C_{\alpha_n} \sum_{\ell=\alpha_nn}^{n} T_\ell(x))( C_{\alpha_n} \sum_{\ell^\prime=\alpha_nn}^{n}  \Delta T_{\ell^\prime} (x) )\,dx\bigg)\\
	&=& \frac{C_{\alpha_n}^2}{D_{\alpha_n}} \sum_{\ell=\alpha_nn}^{n} \sum_{\ell^\prime=\alpha_nn}^{n}   \ell^\prime (\ell^\prime+1) \int_{\mathbb{S}^2}  T_\ell(x)T_{\ell^\prime} (x)\,dx\\
	&=&\frac{C_{\alpha_n}^2}{D_{\alpha_n}} \sum_{\ell=\alpha_n}^{n} \ell(\ell+1) \int_{\mathbb{S}^2} T_\ell^2(x)\, dx.
	\end{eqnarray*}

	Then the second chaotic projection becomes
	
	\begin{equation}
	\begin{split}
\mathcal{L}_{\alpha_{n}}[2]= \sqrt{D_{\alpha_n}} \bigg( \frac{\beta_0 \alpha_{00}}{2} \int_{\mathbb{S}^2}  H_2(T_{\alpha_n}(x))\,dx+\frac{\beta_0 \alpha_{20}}{2} \int_{\mathbb{S}^2} \left(\frac{C_{\alpha_n}^2}{D_{\alpha_n}} \sum_{\ell=\alpha_n}^{n} \ell(\ell+1) \int_{\mathbb{S}^2} T_\ell^2(x)-2  \right)\,dx \bigg) .
	\end{split}
	\end{equation}
	Now we observe that 
	
	\begin{eqnarray*}
\int_{\mathbb{S}^2}  H_2(T_{\alpha_n}(x))\,dx&=& \int_{\mathbb{S}^2}  C_{\alpha_n}^2 \sum_{\ell,\ell^\prime=\alpha_nn}^{n} T_\ell(x) T_{\ell^\prime}(x)-1\,dx= \int_{\mathbb{S}^2}  C_{\alpha_n}^2 \sum_{\ell=\alpha_nn}^{n} T_\ell(x)^2-1\,dx\\
&=&\int_{\mathbb{S}^2}  C_{\alpha_n}^2 \sum_{\ell=\alpha_nn}^{n} T_\ell(x)^2- C_{\alpha_n}^2 \sum_{\ell=\alpha_nn}^{n} \frac{2\ell+1}{4\pi}\,dx=C_{\alpha_n}^2 \sum_{\ell=\alpha_nn}^{n} \frac{2\ell+1}{4\pi}  \int_{\mathbb{S}^2} \frac{T_\ell(x)^2}{Var[T_\ell(x)]}-1 \,dx\\
&=& C_{\alpha_n}^2 \sum_{\ell,\ell^\prime=\alpha_nn}^{n}\frac{2\ell+1}{4\pi}H_2(\tilde{T}_{\ell}(x))\,dx.
	\end{eqnarray*}
	
It follows that
	$$
	\int_{\mathbb{S}^2} \left(\frac{C_{\alpha_n}^2}{D_{\alpha_n}} \sum_{\ell=\alpha_n}^{n} \ell(\ell+1) T_\ell^2(x)-2  \right)\,dx= 2 \int_{\mathbb{S}^2} \left(\frac{C_{\alpha_n}^2}{D_{\alpha_n}} \sum_{\ell=\alpha_n}^{n} \frac{\ell(\ell+1)}{2} T_\ell^2(x)-1 \right)\,dx .$$

	Since $$D_{\alpha_n}= C_{\alpha_n}^2 \sum_{\ell=\alpha_nn}^{n} \frac{\ell(\ell+1)}{2} \frac{2\ell+1}{4\pi},$$ we can write
	$$1= \frac{ C_{\alpha_n}^2}{D_{\alpha_n}} \sum_{\ell=\alpha_nn}^{n} \frac{\ell(\ell+1)}{2}\frac{2\ell+1}{4\pi}$$
	and then
	
	\begin{eqnarray*}
2 \int_{\mathbb{S}^2} \left(\frac{C_{\alpha_n}^2}{D_{\alpha_n}} \sum_{\ell=\alpha_n}^{n} \frac{\ell(\ell+1)}{2} T_\ell^2(x)-1 \right)\,dx&=&2 \int_{\mathbb{S}^2} \frac{C_{\alpha_n}^2}{D_{\alpha_n}} \sum_{\ell=\alpha_n}^{n} \frac{\ell(\ell+1)}{2} \left(T_\ell^2(x)- \frac{2\ell+1}{4\pi} \right)\,dx\\
&=&2 \int_{\mathbb{S}^2} \frac{C_{\alpha_n}^2}{D_{\alpha_n}} \sum_{\ell=\alpha_n}^{n} \frac{\ell(\ell+1)}{2} \frac{2\ell+1}{4\pi} \left(\frac{T_\ell^2(x)}{\frac{2\ell+1}{4\pi}}- 1 \right)\,dx\\
&=&2 \int_{\mathbb{S}^2} \frac{C_{\alpha_n}^2}{D_{\alpha_n}} \sum_{\ell=\alpha_n}^{n} \frac{\ell(\ell+1)}{2} \frac{2\ell+1}{4\pi} \left(\frac{T_\ell^2(x)}{Var[T_\ell(x)]}- 1 \right)\,dx\\
&=& 2  \frac{C_{\alpha_n}^2}{D_{\alpha_n}} \sum_{\ell=\alpha_n}^{n} \frac{\ell(\ell+1)}{2} \frac{2\ell+1}{4\pi} \int_{\mathbb{S}^2} H_2(\tilde{T}_\ell(x) )\,dx,
	\end{eqnarray*}

	where in the last line we write
	
	$$\tilde{T}_\ell(x)=T_\ell(x)/\sqrt{Var[T_\ell(x)]}.$$
	
Hence, the second projection is
	\begin{equation}
	\begin{split}
\mathcal{L}_{\alpha_{n}}[2]&= \sqrt{D_{\alpha_n}} \bigg( \frac{\beta_0 \alpha_{00}}{2} \int_{\mathbb{S}^2}  H_2(T_{\alpha_n}(x))\,dx+\frac{\beta_0 \alpha_{20}}{2} 2  \frac{C_{\alpha_n}^2}{D_{\alpha_n}} \sum_{\ell=\alpha_n}^{n} \frac{\ell(\ell+1)}{2}\frac{2\ell+1}{4\pi} \int_{\mathbb{S}^2} H_2(\tilde{T}_\ell(x) )\,dx\bigg)\\&
	= \frac{\sqrt{D_{\alpha_n}}}{4} C_{\alpha_n}^2 \sum_{\ell=\alpha_nn}^{n} \left( \frac{\ell(\ell+1)}{2D_{\alpha_n}}-1\right) \frac{2\ell+1}{4\pi} \int_{\mathbb{S}^2} H_2(\tilde{T}_\ell(x))\,dx.
	\end{split}
	\end{equation}
Let us compute its variance. We get
	\begin{equation}
	\begin{split}
	Var(\mathcal{L}_{\alpha_{n}}[2])&= \frac{D_{\alpha_n}}{16} C_{\alpha_n}^4 \sum_{\ell=\alpha_nn}^{n} \left( \frac{\ell(\ell+1)}{2D_{\alpha_n}}-1\right) ^2 \left(\frac{2\ell+1}{4\pi}\right)^2 \int_{\mathbb{S}^2\times \mathbb{S}^2 } 2 P_\ell(\cos d(x,y))^2 \,dxdy\\&
	=2\pi^2 D_{\alpha_{n}} C_{\alpha_{n}}^4 \sum_{\ell=\alpha_nn}^{n} \left( \frac{\ell(\ell+1)}{2D_{\alpha_n}}-1\right)^2 \left(\frac{2\ell+1}{4\pi}\right)^2 \int_{0}^{\pi/2}  P_\ell(\cos \theta)^2 \sin \theta\,d\theta\\&
	=2\pi^2 D_{\alpha_{n}} C_{\alpha_{n}}^4 \sum_{\ell=\alpha_nn}^{n} \left( \frac{\ell(\ell+1)}{2D_{\alpha_n}}-1\right)^2 \left(\frac{2\ell+1}{4\pi}\right)^2\frac{1}{2\ell+1}\\&
	= \frac{2\pi^2C_{\alpha_n}^4 D_{\alpha_{n}}}{16\pi^2}  \sum_{\ell=\alpha_nn}^{n}\bigg\{ \frac{(\ell(\ell+1))^2(2\ell+1)}{4D_{\alpha_{n}}^2}+ (2\ell+1)- \frac{\ell(\ell+1)(2\ell+1)}{D_{\alpha_{n}}}\bigg\}\\
	&= \frac{C_{\alpha_n}^4 D_{\alpha_{n}}}{8	\cdot 4 D_{\alpha_{n}}^2}  \sum_{\ell=\alpha_nn}^{n}\bigg\{ (\ell(\ell+1))^2(2\ell+1)+4D_{\alpha_{n}}^2 (2\ell+1)-4D_{\alpha_{n}} \ell(\ell+1)(2\ell+1)\bigg\}.
		\end{split}
	\end{equation}
	
	Noting that $C_{\alpha_{n}}^2=\frac{1}{\sum \frac{2\ell+1}{4\pi}}$ and then $\frac{1}{C_{\alpha_{n}}^2}= \sum \frac{2\ell+1}{4\pi}$
	and $D_{\alpha_{n}}= C_{\alpha_{n}}^2 \sum \frac{\ell(\ell+1)}{2}\frac{2\ell+1}{4\pi},$
	we have that
	
	\begin{eqnarray*}
	Var(\mathcal{L}_{\alpha_{n}}[2])&=&  \frac{C_{\alpha_n}^4 D_{\alpha_{n}}}{8	\cdot 4 D_{\alpha_{n}}^2} \bigg\{ \sum_{\ell=\alpha_nn}^{n} (\ell(\ell+1))^2(2\ell+1)+\frac{4D_{\alpha_{n}}^2}{C_{\alpha_{n}}^2} 4\pi-4D_{\alpha_{n}} \frac{D_{\alpha_{n}}}{C_{\alpha_{n}}^2} 8\pi \bigg\}\\
	&=&  \frac{C_{\alpha_n}^4}{32D_{\alpha_{n}}} \bigg\{ \sum_{\ell=\alpha_nn}^{n} \ell^2(\ell+1)^2(2\ell+1)-16\pi \frac{D_{\alpha_{n}}^2}{C_{\alpha_{n}}^2} \bigg\}.
\end{eqnarray*}

	Since 
	$$\frac{D_{\alpha_{n}}^2}{C_{\alpha_{n}}^2}=\frac{ (C_{\alpha_{n}}^2 \sum_{\alpha_nn}^n \frac{\ell(\ell+1)}{2}\frac{2\ell+1}{4\pi})^2}{C_{\alpha_{n}}^2}=\frac{C_{\alpha_{n}}^2 (\sum_{\alpha_nn}^n \ell(\ell+1)(2\ell+1))^2}{64\pi^2},$$

	we get
	\begin{eqnarray}\label{varproj2}
	Var(\mathcal{L}_{\alpha_{n}}[2])		&=& \frac{C_{\alpha_n}^4}{32D_{\alpha_{n}}} \bigg\{ \sum_{\ell=\alpha_nn}^{n} \ell^2(\ell+1)^2(2\ell+1)-16\pi \frac{C_{\alpha_{n}}^2 (\sum_{\alpha_nn}^n \ell(\ell+1)(2\ell+1))^2}{64\pi^2}\bigg\} \nonumber\\
	&=&\frac{C_{\alpha_n}^6}{32D_{\alpha_{n}}} \bigg\{ \frac{1}{C_{\alpha_{n}}^2}\sum_{\ell=\alpha_nn}^{n} \ell^2(\ell+1)^2(2\ell+1)-16\pi \frac{ (\sum_{\alpha_nn}^n \ell(\ell+1)(2\ell+1))^2}{64\pi^2}\bigg\} .
	\end{eqnarray}
	
From (\ref{C2}) we have
	$$\frac{1}{C_{\alpha_{n}}^2}= \frac{2g(n)n^2-g(n)^2n^2+2n +1}{4\pi}.$$
	
	Now we have to compute the two sums appearing in (\ref{varproj2}). We start by considering
	
	\begin{eqnarray*}
	\sum_{\ell=\alpha_nn}^n \ell(\ell+1)(2\ell+1)&=&\sum_{\ell=\alpha_nn}^n 2\ell^3 +3\ell^2 +\ell= \sum_{\ell=1}^n (2\ell^3 +3\ell^2 +\ell)-\sum_{\ell=1}^{\alpha_nn-1} (2\ell^3 +3\ell^2 +\ell)\\
	&=&\frac{1}{2}\left[ n(n+1)^2(n+2)-(\alpha_{n}n-1)(\alpha_{n}n)^2(\alpha_{n}n+1) \right]\\
	&=& \frac{n^4}{2}\left[ 1+\frac{4}{n}+\frac{5}{n^2}+\frac{2}{n^3}-\alpha^4+\frac{\alpha^2}{n^2} \right].
	\end{eqnarray*}
	
	Replacing $\alpha_{n}=1-g(n)$ we have
	
	\begin{eqnarray*}
		\sum_{\ell=\alpha_nn}^n \ell(\ell+1)(2\ell+1)&=&\frac{n^4}{2}\left[4g(n)- 6g(n)^2+4g(n)^3-g(n)^4+\frac{4}{n}+\frac{6}{n^2}+\frac{2}{n^3}+\frac{g(n)^2}{n^2}-\frac{2 g(n)}{n^2} \right]\\
	&=&\frac{n^44g(n)}{2}\bigg[1- \frac{6}{4} g(n)+g(n)^2-\frac{g(n)^3}{4} +\frac{1}{ng(n)}+\frac{6}{4g(n)n^2}+\frac{2}{4g(n)n^3}\\&&\quad+\frac{g(n)}{4n^2}-\frac{2 }{4n^2} \bigg].
	\end{eqnarray*}

Similarly for the other sum of (\ref{varproj2}) we obtain

	\begin{eqnarray*}
		\sum_{\ell=\alpha_nn}^n &&\ell^2(\ell+1)^2(2\ell+1) = \sum_{\ell=1}^n (\ell^2(\ell+1)^2(2\ell+1))-\sum_{1}^{\alpha_nn-1} (\ell^2(\ell+1)^2(2\ell+1))\\
	&&=\frac{1}{3} n^2(n+1)^2(n+2)^2- \frac{1}{3} (\alpha_nn-1)^2(\alpha_nn)^2(\alpha_nn+1)^2 \\
	&&= \frac{n^6}{3}+2n^5+\frac{13n^4}{3}+4n^3+\frac{4n^2}{3}- \bigg[ \frac{(\alpha_nn)^6}{3}+2(\alpha_nn)^5+\frac{13(\alpha_nn)^4}{3}+4(\alpha_nn)^3+\frac{4(\alpha_nn)^2}{3} \bigg]\\
	&&= \frac{-n^2}{3}(ng(n)-2n-1)(g(n)n+1)(g(n)^2n-3g(n)n-g(n)+3n+3)(g(n)^2n\\
	&&\quad-g(n)n+g(n)+n+1)\\
	&&= \frac{-n^2}{3} \bigg[ n^4g(n)^6-6n^4g(n)^5-2g(n)^4n^2+15g(n)^4n^4+8g(n)^3n^2\\&&\quad-20n^4g(n)^3-12g(n)^2n^2+15g(n)^2n^4+8g(n)n^2
	\\
	&&\quad+2g(n)n+g(n)^2-2g-6n^4g-6n^3-15n^2-12n-3\bigg]\\
	&& = n^2\bigg[ -\frac{n^4g(n)^6}{3}+2n^4g(n)^5+\frac{2}{3}g(n)^4n^2-5g(n)^4n^4-\frac{8}{3}g(n)^3n^2\\&&\quad
	+\frac{20}{3}n^4g(n)^3+4g(n)^2n^2-5g(n)^2n^4-\frac{8}{3}g(n)n^2\\
	&& \quad-\frac{2}{3}g(n)n-\frac{1}{3}g(n)^2+\frac{2}{3}g(n)+2n^4g(n)+2n^3+5n^2+4n+1\bigg]
	\end{eqnarray*}
	
	and hence we get
	
	\begin{eqnarray*}
	&&Var(\mathcal{L}_{\alpha_{n}}[2])=\frac{C_{\alpha_n}^6}{32D_{\alpha_{n}}} \bigg\{ \frac{1}{C_{\alpha_{n}}^2}\sum_{\ell=\alpha_nn}^{n} \ell^2(\ell+1)^2(2\ell+1)-16\pi \frac{ (\sum_{\alpha_nn}^n \ell(\ell+1)(2\ell+1))^2}{64\pi^2}\bigg\} \\
&&=\frac{C_{\alpha_n}^6}{32D_{\alpha_{n}}}\bigg\{ \frac{2g(n)n^2-g(n)^2n^2+2n +1}{4\pi} \sum_{\ell=\alpha_nn}^{n} \ell^2(\ell+1)^2(2\ell+1)-\frac{ (\sum_{\alpha_nn}^n \ell(\ell+1)(2\ell+1))^2}{4\pi}\bigg\}\\
&&= \frac{C_{\alpha_n}^6 n^2}{32D_{\alpha_{n}} 4\pi}\bigg\{ (2g(n)n^2-g(n)^2n^2+2n +1) \bigg[ -\frac{n^4g(n)^6}{3}+2n^4g(n)^5+\frac{2}{3}g(n)^4n^2-5g(n)^4n^4\\	&&\quad -\frac{8}{3}g(n)^3n^2
+\frac{20}{3}n^4g(n)^3+4g(n)^2n^2-5g(n)^2n^4-\frac{8}{3}g(n)n^2-\frac{2}{3}g(n)n-\frac{1}{3}g(n)^2+\frac{2}{3}g(n)\\&&\quad+2n^4g(n)+2n^3+5n^2+4n+1\bigg]\\
	&&\quad -\left[2n^3g(n)- 3g(n)^2n^3+2g(n)^3n^3-\frac{g(n)^4}{2}n^3+2n^2+3n+1+\frac{g(n)^2n}{2}- ng(n) \right]^2\bigg\}\\
	&&= \frac{C_{\alpha_n}^6 n^2}{32D_{\alpha_{n}} 4\pi}\bigg\{ 1+6n+13n^2+12n^3+4n^4+4g(n)^2n^6+8g(n)n^5-12g(n)^2n^5+12g(n)n^4\\
	&&\quad -\frac{46}{3}g(n)^2n^4 +\frac{8}{3}g(n)^2n^3+\frac{8}{3}g(n)n^3+\frac{13}{3}g(n)^2 n^2-2g(n) n^2+\frac{2}{3}g(n)n-\frac{2}{3}g(n)^2n+\frac{2}{3}g(n)-\frac{1}{3}g(n)^2\\
	&&\quad -\left[2n^3g(n)- 3g(n)^2n^3+2g(n)^3n^3-\frac{g(n)^4}{2}n^3+2n^2+3n+1+\frac{g(n)^2n}{2}- ng(n) \right]^2\bigg\}.
\end{eqnarray*}
	
	After cancellations we find
	
	\begin{eqnarray*}
Var(\mathcal{L}_{\alpha_{n}}[2])&=&\frac{C_{\alpha_n}^6 n^2}{32D_{\alpha_{n}} 4\pi}\bigg\{ n^6g(n) \bigg[ 
\frac{8}{3n^3}-\frac{46g(n)}{3n^2}+\frac{52g(n)^2}{3n^2}+\frac{55g(n)^3}{3}+\frac{40g(n)^2}{3n}-17g(n)^3\\&&-\frac{8g(n)^2}{n}-\frac{20g(n)^2}{n^2}+\frac{22g(n)}{n^2}+O\left(\frac{g(n)}{n^3}\right)+O\left(\frac{g(n)^3}{n^2}\right)\bigg]\bigg\}\\
&=&\frac{C_{\alpha_n}^6 n^2}{32D_{\alpha_{n}} 4\pi}\bigg\{ n^6g(n)^3 \bigg[ 
\frac{8}{3n^3g(n)^2}-\frac{46g(n)}{3n^2g(n)^2}+\frac{52g(n)^2}{3n^2g(n)^2}+\frac{55g(n)^3}{3g(n)^2}+\frac{40g(n)^2}{3ng(n)^2}\\&&-17\frac{g(n)^3}{g(n)^2}-\frac{8g(n)^2}{ng(n)^2}-\frac{20g(n)^2}{n^2g(n)^2}+\frac{22g(n)}{n^2g(n)^2} \quad+O\left(\frac{g(n)}{n^3g(n)^2}\right)+O\left(\frac{g(n)^3}{g(n)^2n^2}\right)\bigg]\bigg\}\\
&=&\frac{C_{\alpha_n}^6 n^2}{32D_{\alpha_{n}} 4\pi}\bigg\{ n^6g(n)^3 \bigg[ 
\frac{8}{3n^3g(n)^2}-\frac{46}{3n^2g(n)}+\frac{52}{3n^2}+\frac{55g(n)}{3}+\frac{40}{3n}-17g(n)-\frac{8}{n}-\frac{20}{n^2}\\
&&\quad+\frac{22}{n^2g(n)}+O\left(\frac{1}{n^3g(n)}\right)+O\left(\frac{g(n)}{n^2}\right)\bigg]\bigg\}\\
&=&\frac{C_{\alpha_n}^6 n^2}{32D_{\alpha_{n}} 4\pi}\bigg\{ n^6g(n)^4 \bigg[ 
\frac{8}{3n^3g(n)^3}-\frac{46}{3n^2g(n)^2}+\frac{52}{3n^2g(n)}+\frac{55}{3}+\frac{40}{3ng(n)}-17-\frac{8}{ng(n)}\\&&-\frac{20}{n^2g(n)}+\frac{22}{n^2g(n)^2}+O\left(\frac{1}{n^3g(n)^2}\right)+O\left(\frac{1}{n^2}\right)\bigg]\bigg\}\\
 &=&\frac{C_{\alpha_n}^6}{32D_{\alpha_{n}} 4\pi} n^8 g(n)^4 \bigg[ 
\frac{8}{3n^3g(n)^3}-\frac{46}{3n^2g(n)^2}+\frac{52}{3n^2g(n)}+\frac{55}{3}+\frac{40}{3ng(n)}-17-\frac{8}{ng(n)}-\frac{20}{n^2g(n)}\\&&+\frac{22}{n^2g(n)^2}+O\left(\frac{1}{n^3g(n)^2}\right)+O\left(\frac{1}{n^2}\right)\bigg].
	\end{eqnarray*}
	
	Now we note that $$\frac{C_{\alpha_{n}}^6}{D_{\alpha_{n}}}=\frac{8\pi C_{\alpha_{n}}^4}{ \sum_{\alpha_nn}^n \ell(\ell+1)(2\ell+1)}= \frac{64\pi^3}{n^8g(n)^3} \left(1+\sum_{ k=1}^{\infty} \left(\frac{g(n)}{2}-\frac{1}{ng(n)}-\frac{1}{2n^2g(n)}\right)^k\right)^2$$
	$$\times \left[1+ \sum_{ k=1}^{\infty} \left(\frac{6g(n)}{4}-g(n)^2+\frac{g(n)^3}{4}-\frac{1}{ng(n)}-\frac{6}{4n^2g(n)} -\frac{2}{4g(n)n^3}-\frac{g(n)}{4n^2}+\frac{1}{2n^2}\right)^k\right].$$
	
Finally we get 
	\begin{eqnarray*}
Var(\mathcal{L}_{\alpha_{n}}[2])&=&\frac{1}{32*4\pi} \frac{64\pi^3}{n^8g(n)^3} \left(1+\sum_{ k=1}^{\infty} \left(\frac{g(n)}{2}-\frac{1}{ng(n)}-\frac{1}{2n^2g(n)}\right)^k\right)^2\\
&&\quad \times \left[1+ \sum_{ k=1}^{\infty} \left(\frac{6g}{4}-g(n)^2+\frac{g(n)^3}{4}-\frac{1}{ng(n)}-\frac{6}{4n^2g(n)} -\frac{2}{4g(n)n^3}-\frac{g(n)}{4n^2}+\frac{1}{2n^2}\right)^k\right]\\
&&\quad \times n^8g(n)^4 \bigg[ 
\frac{8}{3n^3g(n)^3}-\frac{46}{3n^2g(n)^2}+\frac{52}{3n^2g(n)}+\frac{4}{3}+\frac{40}{3ng(n)}-\frac{8}{ng(n)}-\frac{20}{n^2g(n)}\\&&\quad+\frac{22}{n^2g(n)^2}+O\left(\frac{1}{n^3g(n)^2}\right)+O\left(\frac{1}{n^2}\right)\bigg]\\
&=&\frac{1}{32*4\pi} \frac{64\pi^3}{n^8g(n)^3}  n^8g(n)^4 \frac{4}{3} \bigg[1+ \frac{g(n)}{2}-\frac{1}{ng(n)}+\frac{6}{4}g(n)-\frac{1}{ng(n)} +O\left(g(n)^2\right)\\&&\quad+O\left(\frac{1}{n^2g(n)^2}\right)+O\left(\frac{1}{n}\right)\bigg]\\
&=& \frac{2\pi^2}{3} g(n) \bigg[1+2g(n)-\frac{2}{n g(n)} +O\left(g(n)^2\right)+O\left(\frac{1}{n^2g(n)^2}\right) + O\left(\frac{1}{n}\right)\bigg],
	\end{eqnarray*}
which completes the proof.
	
\end{proof}

\appendix

\section{Asymptotics for the covariance kernel and its derivatives}\label{asympcovsection}

\subsection{asymptotic for the covariance function}
The following lemma gives the
asymptotic behaviour in the high frequency limit of the covariance function
for $C<\psi <\alpha_{n} m \pi/2$, for $C>0$.

\begin{lemma}
	\label{Lemma1}  Given ${\Gamma}_{\alpha_{n}}(x,y)$ as in (\ref{cov}) and denoting
	\begin{equation}\label{h}
	h:= \sum_{ k=1}^{\infty} g(n)^k+ \frac{1}{n}  +\frac{1}{2n} \sum_{k=1}^{\infty} g(n)^k, 
	\end{equation}

	we have that, for $C < \psi < \alpha_n m\pi/2$, $C>0$,
	as $n\to \infty$,	
	\begin{equation}  \label{asymptotic cov function}
	\begin{split}
{\Gamma}_{\alpha_{n}} \left(\cos \frac{\psi}{\alpha_{n} m}\right)&=
	\frac{C_{\alpha_{n}}^2\sqrt{n}}{4\pi \sqrt{\pi}}  \bigg(\sin \frac{\psi}{2\alpha_{n} m}\bigg)^{-3/2}  \left(\cos \frac{\psi}{2\alpha_{n}m}\right)^{-1/2} \frac{\psi h}{2} \\&\bigg[2\sin\left(\frac{h\psi}{2}+\psi-\frac{\psi}{2n}+O\left(\frac{\psi}{2n}\right)+\frac{\pi}{4}\right) \left(\frac{\sin \psi h}{\psi h}\right)+\frac{g(n)}{\psi h} \cos\left(\psi-\frac{\psi}{2n}+O\left(\frac{\psi}{2n}\right)-\frac{3}{4}\pi\right)\\&+O\left(\frac{1}{g(n)n\psi}+\frac{g(n)}{\psi}\right)\bigg]+O\left(\frac{1}{n^2g(n)\sqrt{n}}\right).
	\end{split}
	\end{equation}
	
	\begin{corollary}\label{gamma} For $C < \psi < \alpha_n m\pi/2$, $C>0$,
		as $n\to \infty$,	 we have,
	$${\Gamma}_{\alpha_{n}} \left(\cos \frac{\psi}{\alpha_{n} m}\right)= \sqrt{\frac{2}{\pi\psi}} \bigg\{ \sin\left(\frac{h\psi}{2}+\psi-\frac{\psi}{2n}+O(\frac{\psi}{2n})+\frac{\pi}{4}\right)+\frac{1}{2\psi} \cos\left(\psi-\frac{\psi}{2n}+O\left(\frac{\psi}{n^2}\right)-\frac{3}{4}\pi\right)$$$$+O\left(\frac{1}{ng(n)\psi}+\frac{g(n)}{\psi}\right) \bigg\}
+O\left(\frac{g(n)}{\sqrt{\psi}}\right)+O\left( \frac{1}{\sqrt{\psi}ng(n)}\right)+O\left(\frac{1}{n^2g(n)\sqrt{n}}\right).$$
	\end{corollary}
	
\end{lemma}

\begin{corollary}
For $C < \psi < \alpha_n m\pi/2$, $C>0$,
as $n\to \infty$,	 we have,	
	\begin{eqnarray*}
{\Gamma}_{{\alpha_{n}}} \left(\cos \frac{\psi}{\alpha_{n}m}\right)^2&=&
	\frac{C_{\alpha_{n}}^4n}{(4\pi)^2 \pi}  \bigg(\sin \frac{\psi}{2\alpha_{n} m}\bigg)^{-3}  \left(\cos \frac{\psi}{2\alpha_{n}m}\right)^{-1} \frac{\psi^2 h^2}{4}\\&&
	\bigg[\left(2+2\sin\left(h\psi+2\psi-\frac{\psi}{n}+O\left(\frac{\psi}{n^2}\right)\right)\right) \left(\frac{\sin h\psi}{h\psi}\right)^2\\&&+4\frac{g(n)}{\psi h} \left(\frac{\sin h\psi}{h\psi}\right)   \sin\left(\frac{h\psi}{2}+\psi-\frac{\psi}{2n}+O(\frac{\psi}{n^2})+\frac{\pi}{4}\right)\cos\left(\psi-\frac{\psi}{2n}+O\left(\frac{\psi}{n^2}\right)-\frac{3}{4}\pi\right)\\&&+O\left(\frac{1}{g(n)n\psi}+\frac{g(n)}{\psi}\right)+O\left(\frac{1}{\psi^2}\right)\bigg]+O\left(\frac{1}{n^{5/2}g(n)\sqrt{\psi}}\right).
	\end{eqnarray*}

\end{corollary}

Exploiting the fact that $\sin x \sim x$ and $C^2=\frac{4\pi}{2n^2g(n)} (1+O(g(n)))$ we obtain the following corollary.
\begin{corollary}\label{gamma-quadro}
For $C < \psi < \alpha_n m\pi/2$, $C>0$,
as $n\to \infty$,	  we have,
	\begin{eqnarray*}
	\Gamma_{\alpha_{n}} \left(\cos \frac{\psi}{\alpha_n m}\right)^2&=&\frac{2}{\pi\psi} \bigg\{ \left(\frac{1}{2}+\frac{1}{2}\sin\left(h\psi+2\psi-\frac{\psi}{n}+O\left(\frac{\psi}{n^2}\right)\right)\right)+ \frac{g(n)}{\psi h} \sin\left(\frac{h\psi}{2}+\psi-\frac{\psi}{2n}+O(\frac{\psi}{n^2})+\frac{\pi}{4}\right)\\&&\times\cos\left(\psi-\frac{\psi}{2n}+O\left(\frac{\psi}{n^2}\right)\right)
	\\&&+O\left(\frac{1}{\psi^2}\right)+O\left(\frac{1}{g(n)n\psi}+\frac{g(n)}{\psi}\right) \bigg\}+ O\left(\frac{g(n)}{\psi}\right)+O\left(\frac{1}{ng(n)\psi}\right).
\end{eqnarray*}
\end{corollary}

\begin{proof}[\bf{Proof of Lemma \ref{Lemma1}}]
	Looking at the covariance function in (\ref{cov}) we can write ${\Gamma}_{\alpha_{n}}(x,y)$ as \begin{equation}\label{var1}
	\begin{split}
{\Gamma}_{\alpha_{n}}(x,y)&= 
	C_{\alpha_{n}} \bigg(\sum_{\ell=0}^n \frac{2\ell+1}{4\pi} P_{\ell}(\langle x,y \rangle)-\sum_{\ell=0}^{n\alpha_{n}-1} \frac{2\ell+1}{4\pi} P_{\ell}(\langle x,y \rangle )\bigg).
	\end{split}
	\end{equation}
	
(see also \cite{Todino3}).
	Thanks to the following formula (\cite{FXA}, page 6), derived by the Christoffel-Darboux formula (\cite{AH12}), 
	\begin{equation}\label{CDformula}
	\sum_{\ell=0}^{n} \sum_{m=-\ell}^{\ell} Y_{\ell, m} (x) Y_{\ell, m} (y)= \frac{n+1}{4\pi} P_n^{(1,0)}(\cos \theta(x,y)),
	\end{equation} 
	where $P_n(\cos \theta)=P_n^{(1,0)}(\cos \theta)$,
	and to the addition formula (\cite{M e Peccati} page 66):
	\begin{equation}\label{addition formula}
	\sum_{m=-\ell}^{\ell} Y_{\ell, m} (x) Y_{\ell, m} (y)= \frac{2\ell+1}{4\pi} P_\ell(\cos \theta(x,y)),
	\end{equation}
	we obtain that 
	\begin{equation}\label{Gamma}
	{\Gamma}_{\alpha_{n}}(\cos \theta)=C_{\alpha_{n}}^2 \bigg[ \frac{n+1}{4\pi} P_n^{(1,0)}(\cos \theta(x,y))-\frac{n\alpha_{n}}{4\pi} P_{n\alpha_{n}-1}^{(1,0)} (\cos \theta(x,y)) \bigg].
	\end{equation}
	
	Exploiting Theorem 8.21.13 \cite{szego}, 
	
	\begin{equation}\label{CD}
P_n^{\alpha,\beta}(\cos \theta)=n^{-1/2} k(\theta) \cos(N\theta+\gamma)+O(n^{-1/2}),
	\end{equation}
	
	$$k(\theta)=\pi^{-1/2} (\sin \theta/2)^{-\alpha-1/2}(\cos \theta/2)^{-\beta-1/2},$$
	
	$$N=n+(\alpha+\beta+1)/2, \quad \gamma=-(\alpha+1/2)\pi/2;$$
	
we get
	
		\begin{equation}
	\begin{split}
	{\Gamma}_{\alpha_{n}}(\cos \theta)&=\frac{C_{\alpha_{n}}^2}{4\pi\sqrt{\pi}} \bigg[ (n+1) \left( n^{-1/2} \left(\sin \frac{\theta}{2}\right)^{-3/2} (\cos \theta)^{-1/2}  \cos\left((n+1)\theta-\frac{3}{4}\pi\right)+O\left(\frac{1}{\sqrt{n}}\right)\right) \\&-n{\alpha_{n}} \left( (n\alpha-1)^{-1/2}(\sin \theta/2)^{-3/2} \cos(\theta/2)^{-1/2} \cos\left(n\alpha\theta-\frac{3}{4}\pi\right) +O\left(\frac{1}{\sqrt{\alpha n-1}}\right) \right) \bigg].
	\end{split}
	\end{equation}
	
Expanding $(1+1/n)$ and  $(1-g(n))$ (since $g(n)\to 0$ as $n \to \infty$)

\[ {\Gamma}_{\alpha_{n}}(\cos \theta)= \frac{\sqrt{n}}{4\pi \sqrt{\pi}} C^2_{\alpha_{n}} \left(\sin \theta/2\right)^{-3/2} \left(\cos \theta/2\right)^{-1/2} \bigg[ \cos\left((n+1)\theta-\frac{3}{4}\pi\right)+\frac{1}{n} \cos\left((n+1)\theta-\frac{3}{4}\pi\right)\]

\[ -(1-g(n))\left(1+\frac{g(n)}{2}+\frac{1}{2n}+\frac{3}{8g(n)^2}+\frac{3}{8n^2}+O(g(n)^3)\right) \cos\left(n\alpha \theta-\frac{3}{4}\pi\right)  \bigg]+O\left(\frac{1}{n^2g(n)\sqrt{n}}\right) .\]
	
	We change variable $\theta=\psi/(\alpha_{n} m)$ to get
	
	\begin{eqnarray*}
		{\Gamma}_{\alpha_{n}}\left(\cos \frac{\psi}{\alpha_{n}m}\right)&=& \frac{\sqrt{n}}{4\pi \sqrt{\pi}} C^2_{\alpha_{n}} \left(\sin \frac{\psi}{2\alpha_{n}m}\right)^{-3/2} \left(\cos \frac{\psi}{2\alpha_{n}m}\right)^{-1/2} \bigg[ \cos\left((n+1)\frac{\psi}{2\alpha_{n}m}-\frac{3}{4}\pi\right)\\
			&& -\cos\left(\psi\frac{n}{m}-\frac{3}{4}\pi\right)+\frac{g(n)}{2}\cos\left(\psi\frac{n}{m}-\frac{3}{4}\pi\right)\\&&
			+\frac{1}{n} \bigg[ \cos\left((n+1)\frac{\psi}{2\alpha_{n}m}-\frac{3}{4}\pi\right)-\frac{1}{2}\cos\left(\psi \frac{n}{m}-\frac{3}{4}\pi\right)\bigg]+O(g(n)^2)\bigg]+O\left(\frac{1}{n^2g(n)\sqrt{n}}\right).
	\end{eqnarray*}
	
	Now we note that 
	
\[ \frac{n\psi}{m}= \psi-\frac{\psi}{2n}+O\left(\frac{\psi}{n^2}\right) \]

and

\[ \frac{(n+1)\psi}{\alpha_{n}m}=  \psi-\frac{\psi}{2n}+O\left(\frac{\psi}{n^2}\right) +h\psi, \]
	where $h$ is defined in (\ref{h}).
Substituting these expressions and using the fact that 

\begin{eqnarray*}
\cos\left((n+1)\frac{\psi}{2\alpha_{n}m}-\frac{3}{4}\pi\right) -\cos\left(\psi\frac{n}{m}-\frac{3}{4}\pi\right) &=& -2\sin\left(\psi-\frac{3}{4}\pi-\frac{\psi}{2n}+O\left(\frac{\psi}{n^2}\right)+\frac{h\psi}{2}\right)\sin\left(\frac{h\psi}{2}\right)\\&=& 2\sin\left(\frac{h\psi}{2}+\psi-\frac{\psi}{2n}+O\left(\frac{\psi}{n^2}\right)+\frac{\pi}{4}\right)
\sin \frac{h\psi}{2},
\end{eqnarray*}

we finally get
	
	\begin{eqnarray*}
	{\Gamma}_{\alpha_{n}}\left(\cos \frac{\psi}{\alpha_{n}m}\right)&=& \frac{\sqrt{n}}{4\pi \sqrt{\pi}} C_{\alpha_{n}}^2 \left( \sin \frac{\psi}{2\alpha_{n}m}\right)^{-3/2} \left(\cos \frac{\psi}{2\alpha_{n}m }\right)^{-1/2} \psi h\frac{1}{2} \\&& \times \bigg\{ 2\sin\left(\frac{h\psi}{2}+\psi-\frac{\psi}{2n}+O\left(\frac{\psi}{n^2}\right)+\frac{\pi}{4}\right) \left(\frac{\sin h\psi}{h\psi}\right)+\frac{g(n)}{h\psi}\\&& \times \cos\left(\psi-\frac{\psi}{2n}+O\left(\frac{\psi}{n^2}\right)-\frac{3}{4}\pi\right) +O\left(\frac{1}{g(n)n\psi}+\frac{g(n)}{\psi}\right)\bigg\}+O\left(\frac{1}{n^2g(n)\sqrt{n}}\right).
	\end{eqnarray*}
	
\end{proof}

\subsection{Asymptotic for the first derivative of the covariance function}

\begin{lemma}\label{derivata-prima}	For $C < \psi < \alpha_n m\pi/2$, $C>0$,
 as $n\to \infty$,	\begin{eqnarray*}
{\Gamma}_{\alpha_{n}}^\prime \left(\cos \frac{\psi}{\alpha_{n}m}\right)&=&C_{\alpha_{n}}^2\frac{n\sqrt{n}}{8\pi \sqrt{\pi} } \left(-\sin\frac{\psi}{\alpha_{n}m}\right) \left(\sin \frac{\psi}{2\alpha_{n}m}\right)^{-5/2} \left(\cos \frac{\psi}{2\alpha_{n}m}\right)^{-3/2}  \psi h \\&& \times \bigg[ 2 \frac{\sin \frac{h\psi}{2}}{h\psi} \sin\left(\frac{\psi h}{2}+\psi-\frac{\psi}{2n}+O\left(\frac{\psi}{n^2}\right)-\frac{\pi}{4}\right)+\frac{3g(n)}{2\psi h}\cos \left(\psi-\frac{\psi}{2n}+O\left(\frac{\psi}{n^2}\right)-\frac{5}{4}\pi\right) \\&&-\frac{3}{8}g(n)^2\frac{1}{\psi h} \cos\left(\psi-\frac{\psi}{2n}+O\left(\frac{\psi}{n^2}\right)-\frac{5}{4}\pi\right) + O\left(\frac{1}{n\psi }+\frac{g(n)^2}{\psi}\right)\bigg] \quad +O\left(\frac{\psi}{n^3\sqrt{n}g(n)}\right).
\end{eqnarray*}
\end{lemma}

\begin{corollary}For $C < \psi < \alpha_n m\pi/2$, $C>0$,
	as $n\to \infty$,	
	\begin{eqnarray*}
	{\Gamma}_{\alpha_{n}}'\left(\cos \frac{\psi}{\alpha_n m}\right)^2&=& \frac{C_{\alpha_n}^4}{64\pi^3} n^3h^2 \psi^2 \left(\sin \frac{\psi}{\alpha_{n}m}\right)^2\left( \sin\frac{\psi}{2\alpha_n m}\right)^{-5} \left( \cos \frac{\psi}{\alpha_{n}m}\right)^{-3} \\&&\bigg\{ \left(\frac{\sin \frac{\psi h}{2}}{\psi h/2}\right)^2 \left(\frac{1}{2}-\frac{1}{2} \sin\left(h\psi+2\psi-\frac{\psi}{n}+O\left(\frac{\psi}{n^2}\right)\right)\right) 
	\\&&+\frac{3g(n)}{\psi h } \sin\left(\frac{h\psi}{2}+\psi-\frac{\psi}{2n}+O\left(\frac{\psi}{n^2}\right)-\frac{\pi}{4}\right)\cos\left(\psi-\frac{\psi}{2n}+O\left(\frac{\psi}{n^2}\right)-\frac{5}{4}\pi\right) \frac{\sin(\psi h /2)}{\psi h/2}\\&&-\frac{3}{4\psi h}g(n)^2 \sin\left(\frac{h}{2}+\psi-\frac{\psi}{2n}+O\left(\frac{\psi}{n^2}\right)-\frac{\pi}{4}\right)\cos\left(\psi-\frac{1}{2n}+O\left(\frac{1}{n^2}\right)-\frac{\pi}{4}\right)\\&&+ O\left(\frac{g(n)^2}{\psi}\right)+O\left(\frac{1}{\psi n}\right)+O\left(\frac{1}{\psi^2}\right) \bigg\}+O\left(\frac{\sqrt{\psi}}{n^2\sqrt{n}g(n)}\right).
	\end{eqnarray*}
\end{corollary}

\begin{corollary}\label{gamma-primo-quadro} For $C < \psi < \alpha_n m\pi/2$, $C>0$,
	as $n\to \infty$,	
	\begin{eqnarray*}
	\Gamma_{\alpha_{n}}'\left(\cos \frac{\psi}{\alpha_n m}\right)^2&=&  \frac{2n^2}{\pi\psi}\bigg\{  \left(\frac{1}{2}-\frac{1}{2} \sin\left(h\psi+2\psi-\frac{\psi}{n}+O\left(\frac{\psi}{n^2}\right)\right)\right) 
	\\&&+\frac{3g(n)}{\psi h } \sin\left(\frac{h\psi}{2}+\psi-\frac{\psi}{2n}+O\left(\frac{\psi}{n^2}\right)-\frac{\pi}{4}\right)\cos\left(\psi-\frac{\psi}{2n}+O\left(\frac{\psi}{n^2}\right)-\frac{5}{4}\pi\right) \\&&-\frac{3}{4\psi h}g(n)^2 \sin\left(\frac{h}{2}+\psi-\frac{\psi}{2n}+O\left(\frac{\psi}{n^2}\right)-\frac{\pi}{4}\right)\cos\left(\psi-\frac{1}{2n}+O\left(\frac{1}{n^2}\right)-\frac{\pi}{4}\right) \\&&
	+ O\left(\frac{g(n)^2}{\psi}\right)+O\left(\frac{1}{\psi n}\right)+O\left(\frac{1}{\psi^2}\right) \bigg\}+O\left(\frac{\sqrt{\psi}}{n^2\sqrt{n}g(n)}\right).
\end{eqnarray*}
\end{corollary}

\begin{proof}[\bf{Proof of Lemma \ref{derivata-prima}}]
	
	From (\ref{Gamma}), we get
	
	$${\Gamma}_{\alpha_{n}}^\prime(\cos \theta)=C_{\alpha_{n}}^2 \bigg[ \frac{n+1}{4\pi} \frac{d}{d\theta}P_n^{(1,0)}(\cos \theta(x,y))-\frac{n\alpha_{n}}{4\pi} \frac{d}{d\theta}P_{n\alpha_{n}-1}^{(1,0)} (\cos \theta(x,y)) \bigg].$$
	
		From \cite{szego} (4.5.5) we also know that
		\begin{equation}\label{formula}
\frac{d}{dx}P_n^{a,b}(x)=\frac{1}{2}(a+b+n+1) P_{n-1}^{(a+1,b+1)}(x)
		\end{equation}
	
	and hence 
	
		$${\Gamma}_{\alpha_{n}}^\prime(\cos \theta)=\frac{C_{\alpha_{n}}^2}{8\pi} (-\sin \theta) \bigg[ (n+1)(n+2) P_{n-1}^{(2,1)}(\cos \theta)-n\alpha_{n} (n\alpha_{n}+1)P_{n\alpha_{n}-2}^{(2,1)} (\cos \theta) \bigg].$$
	
Now we exploit (\ref{CD}) to get 
	
	\begin{eqnarray*}
{\Gamma}_{\alpha_{n}}^\prime (\cos \theta)&=& C_{\alpha_{n}} ^2  (-\sin\theta) \bigg\{ (\sin \theta/2)^{-5/2} (\cos \theta/2)^{-3/2} \frac{1}{8\pi \sqrt{\pi} } \bigg[  \frac{(n+1)(n+2)}{\sqrt{n-1}} \cos\left((n+1)\theta-\frac{5}{4}\pi\right)\\&&-\frac{n\alpha_n(n\alpha_n+1)}{\sqrt{n\alpha_n-2}} \cos \left(n\alpha_n \theta-\frac{5}{4}\pi\right)\bigg] +O(\sqrt{1/n}) \bigg\}.
	\end{eqnarray*}
	
	We note that
	$$\frac{(n+1)(n+2)}{\sqrt{n-1}}= n\sqrt{n}\bigg[1+\frac{7}{2n}+\frac{31}{8n^2}+O\left(\frac{1}{n^3}\right)\bigg]$$
	and
	
	$$\frac{n^2\alpha_n(\alpha_n+1/n)}{\sqrt{n} \sqrt{\alpha_n-2/n}}=n\sqrt{n} \left[1+\frac{2}{n}+\frac{5}{2n^2}-\frac{3}{2}g(n)-\frac{g(n)}{n}+\frac{3}{8}g(n)^2+O(g(n)^3)\right].$$
	
Replacing these two expressions above, we find that
	
	\begin{eqnarray*}
{\Gamma}_{\alpha_{n}}^\prime (\cos \theta)&=&C_{\alpha_{n}} ^2 (-\sin\theta) \bigg\{ (\sin \theta/2)^{-5/2} (\cos \theta/2)^{-3/2} \frac{1}{8\pi \sqrt{\pi} } n\sqrt{n} \\&&\quad \times \bigg[ \cos\left((n+1)\theta-\frac{5}{4}\pi\right)-\cos \left (n\alpha_n\theta-\frac{5}{4}\pi\right) +\frac{3}{2}g(n) \cos\left(n\alpha_n\theta-\frac{5}{4}\pi\right)\\&&\quad +\frac{7}{2n} \cos\left((n+1)\theta-\frac{5}{4}\pi\right)-\frac{2}{n} \cos\left(n\alpha_n\theta-\frac{5}{4}\pi\right)-\frac{3}{8}g(n)^2\cos\left(n\alpha_n\theta-\frac{5}{4}\pi\right)   \\&&\quad +O\left(\frac{g(n)}{n} +g(n)^3\right)\bigg]+O\left(\frac{1}{\sqrt{n}}\right)\bigg\}.
	\end{eqnarray*}

	Changing $\theta=\frac{\psi}{\alpha_nm}$, we obtain
	
	\begin{eqnarray*}
{\Gamma}_{\alpha_{n}}^\prime \left(\cos \frac{\psi}{\alpha_{n}m}\right)&=&C_{\alpha_{n}}^2\frac{n\sqrt{n}}{8\pi \sqrt{\pi} } \left(-\sin\frac{\psi}{\alpha_{n}m}\right) \left(\sin \frac{\psi}{2\alpha_{n}m}\right)^{-5/2} \left(\cos \frac{\psi}{2\alpha_{n}m}\right)^{-3/2}  \psi h \\&& \times \bigg[ 2 \frac{\sin \frac{h\psi}{2}}{h\psi} \sin\left(\frac{\psi h}{2}+\psi-\frac{\psi}{2n}+O\left(\frac{\psi}{n^2}\right)-\frac{\pi}{4}\right)+\frac{3g(n)}{2\psi h}\cos \left(\psi-\frac{\psi}{2n}+O\left(\frac{\psi}{n^2}\right)-\frac{5}{4}\pi\right) \\&&-\frac{3}{8}g(n)^2\frac{1}{\psi h} \cos\left(\psi-\frac{\psi}{2n}+O\left(\frac{\psi}{n^2}\right)-\frac{5}{4}\pi\right) + O\left(\frac{1}{n\psi}+\frac{g(n)^2}{\psi}\right)\bigg] \quad +O\left(\frac{\psi}{n^3\sqrt{n}g(n)}\right).
	\end{eqnarray*}
\end{proof}

\subsection{Asymptotic for the second derivative of the covariance function}

\begin{lemma}\label{derivata-seconda}
For $C < \psi < \alpha_n m\pi/2$, with $C>0$,
as $n\to \infty$,	we have
	
\begin{eqnarray*}
	{\Gamma}_{\alpha_{n}}^{''}\left(\cos \frac{\psi}{\alpha_{n}m}\right)&=& \frac{C_{\alpha_{n}}^2}{8\pi \sqrt{\pi}}  \bigg\{ \left(\sin \frac{\psi}{2\alpha_n m}\right)^{-7/2}\left(\cos \frac{\psi}{\alpha_{n}m}\right)^{-5/2} \sin^2 \frac{\psi}{\alpha_n m} \frac{n^2\sqrt{n}}{2}\bigg[-2 \sin \frac{h\psi}{2}\\&& \quad  \sin\left(\frac{h\psi}{2}+\psi-\frac{\psi}{2n}+O\left(\frac{\psi}{n^2}\right)+\frac{\pi}{4}\right) +\frac{5}{2}g(n)\cos\left(\psi-\frac{\psi}{2n}+O\left(\frac{\psi}{n^2}\right)-\frac{7}{4}\pi\right)\\&&
	+O\left(\frac{1}{n}+g(n)^2\right)\bigg]-\sin\left(\psi/2\alpha_{n}m\right)^{-5/2} \left(\cos \psi/2\alpha_{n}m\right)^{-3/2} \cos \frac{\psi}{\alpha_n m}\\
	&&\quad \times n\sqrt{n}\bigg[2\sin \frac{h\psi}{2} \sin\left(\frac{h\psi}{2}+\psi-\frac{\psi}{2n}+O\left(\frac{\psi}{n^2}\right)-\frac{\pi}{4}\right)+\frac{3}{2}g(n) \\&& \quad \times\cos \left(\psi-\frac{\psi}{2n}+O\left(\frac{\psi}{n^2}\right)-\frac{5}{4}\pi\right)-\frac{3}{8}g(n)^2\cos\left(\psi-\frac{\psi}{2n}+O\left(\frac{\psi}{n^2}\right)-\frac{5}{4}\pi\right)\\&& +O\left(\frac{1}{n}\right)+O\left(\frac{g(n)}{n}+g(n)^3\right)\bigg]+O\left(n\sqrt{n}+\psi^2 \sqrt{n}\right) \bigg\}.
\end{eqnarray*}

\end{lemma}

\begin{corollary}
For $C < \psi < \alpha_n m\pi/2$, $C>0$,
as $n\to \infty$,		
	
	\begin{eqnarray*}
	\Gamma_{\alpha_{n}}''\left(\cos \frac{\psi}{\alpha_n m}\right) &=& \sqrt{\frac{2}{\pi\psi} }n^2 \frac{h}{g(n)}\bigg\{ -\sin\left(\frac{h\psi}{2}+\psi-\frac{\psi}{2n}+O\left(\frac{\psi}{n^2}\right)+\frac{\pi}{4}\right)+\frac{5g(n)}{2 h \psi}\\&& \times \cos\left(\psi-\frac{\psi}{2n}+O\left(\frac{\psi}{n^2}\right)-\frac{7}{4}\pi\right)
 -\frac{1}{\psi} \sin\left(\frac{h\psi}{2}+\psi-\frac{\psi}{2n}+O\left(\frac{\psi}{n^2}\right)-\frac{\pi}{4}\right) \\&&+O\left(\frac{1}{n\psi g(n)}\right)+O\left(\frac{g(n)}{\psi}\right) \bigg\}+O\left(\frac{n^2}{g(n)n\sqrt{\psi}}\right) .
\end{eqnarray*}
\end{corollary}

\begin{proof}[\bf{Proof of Lemma \ref{derivata-seconda}}]
	
We have  that
	$$\bar{\Gamma}_{\alpha_{n}}^{''}(\cos \theta)=C_{\alpha_{n}}^2\bigg[ \frac{n+1}{4\pi} P_n^{''}(\cos \theta(x,y))-\frac{n\alpha_{n,\beta}}{4\pi} P_{n\alpha_{n,\beta}-1}^{''} (\cos \theta(x,y)) \bigg].$$
	
We exploit again (\ref{formula}) to derive

	\[
	P_n^{''}(\cos \theta)= \frac{n+2}{2} \bigg\{ \frac{1}{2} (3+n) P_{n-2}^{(3,2)}(\cos \theta) \sin^2\theta-P_{n-1}^{(2,1)}(\cos \theta)\cos \theta \bigg\}
	.\]

Then we get

\begin{eqnarray*}
\Gamma_{\alpha_{n}}(\cos \theta)''&=& C_{\alpha_{n}}^2\bigg\{ \frac{(n+1)(n+2)}{8\pi}\bigg[\frac{(n+3)}{2}P_{n-2}^{(3,2)}(\cos \theta) \sin^2\theta-P_{n-1}^{(2,1)}(\cos\theta) \cos \theta\bigg]\\
&& -\frac{n\alpha_{n}(n\alpha_{n}+1)}{8\pi} \bigg[ \frac{n\alpha_{n}+2}{2}P_{n\alpha_{n}-3}^{(3,2)}(\cos \theta) \sin^2 \theta-P_{n\alpha_{n}-2}^{(2,1)}(\cos\theta) \cos \theta\bigg] \bigg\}.
\end{eqnarray*}

From (\ref{CD}) it follows that
	
\begin{eqnarray*}
\Gamma''_{\alpha_{n}}(\cos \theta)&=& 
\frac{C_{\alpha_{n}}^2}{8\pi\sqrt{\pi}} \bigg\{ \frac{(n+1)(n+2)(n+3)}{2 \sqrt{n-2}} \bigg[ (\sin \theta/2)^{-3-1/2}(\cos \theta/2)^{-2-1/2} \cos \left( (n+1)\theta-\frac{7}{4}\pi\right) +O(1)\bigg]\\&& \times \sin^2\theta - \frac{n\alpha_{n}(n\alpha_{n}+1)(n\alpha_{n}+2)}{2\sqrt{n\alpha_{n}-3}} \bigg[ (\sin\theta/2)^{(-3-1/2)}(\cos\theta/2)^{-2-1/2} \cos\left(n\alpha_{n}\theta-\frac{7}{4}\pi\right)+O(1)\bigg]\\&& \times  \sin^2\theta + \frac{n\alpha_{n}(n\alpha_{n}+1)}{\sqrt{n\alpha_{n}-2}} \bigg[ (\sin\theta/2)^{(-2-1/2)}(\cos\theta/2)^{-1-1/2} \cos\left(n\alpha_{n}\theta-\frac{5}{4}\pi\right)+O(1) \bigg]\cos \theta \\
&& -\frac{(n+1)(n+2)}{\sqrt{n-1}} \left[ (\sin\theta/2)^{(-2-1/2)}(\cos\theta/2)^{-1-1/2} \cos\left((n+1)\theta-\frac{5}{4}\pi\right)+O(1)\right]\cos \theta
\bigg\}\\
&=&
\frac{C_{\alpha_{n}}^2}{8\pi\sqrt{\pi}} \bigg\{ (\sin \theta/2)^{-7/2}(\cos \theta/2)^{-5/2} \sin^2\theta\frac{1}{2} \\&& \times \bigg[ \frac{(n+1)(n+2)(n+3)}{ \sqrt{n-2}} \cos \left( (n+1)\theta-\frac{7}{4}\pi\right) - \frac{n\alpha_{n}(n\alpha_{n}+1)(n\alpha_{n}+2)}{\sqrt{n\alpha_{n}-3}}   \cos\left(n\alpha_{n}\theta-\frac{7}{4}\pi\right)\bigg]\\
&& + O(\sin^2\theta n^2\sqrt{n}) + (\sin\theta/2)^{(-5/2)}(\cos\theta/2)^{-3/2}\cos\theta \\&&
\times \bigg[ \frac{n\alpha_{n}(n\alpha_{n}+1)}{\sqrt{n\alpha_{n}-2}}  \cos\left(n\alpha_{n}\theta-\frac{5}{4}\pi\right) -\frac{(n+1)(n+2)}{\sqrt{n-1}} \cos\left((n+1)\theta-\frac{5}{4}\pi\right)\bigg]+O(n\sqrt{n})
\bigg\}.
\end{eqnarray*}

Now we have that
\begin{eqnarray*}
\frac{n\alpha_{n}(n\alpha_{n}+1)}{\sqrt{n\alpha_{n}-2}} &=& \frac{n^2}{\sqrt{n}} \left( 1-\frac{3g(n)}{2}+\frac{2}{n}-\frac{g(n)}{n}+\frac{3}{8}g(n)^2+\frac{5}{2n^2}+O(g(n)^3)\right) ,\\
\frac{(n+1)(n+2)}{\sqrt{n-1}} &=&\frac{n^2}{\sqrt{n}} \left(1+\frac{7}{2n}+\frac{31}{n^2}+O\left(\frac{1}{n^3}\right)\right) , \\
\frac{n\alpha_{n}(n\alpha_{n}+1)(n\alpha_{n}+2)}{2\sqrt{n\alpha_{n}-3}}&= &\frac{n^3}{\sqrt{n}} \left(1-\frac{5}{2}g(n)+\frac{9}{2n}-\frac{27g(n)}{4n}+\frac{15}{8}g(n)^2+\frac{79}{8n^2}+O(g(n)^3)\right) , \\
	\frac{(n+1)(n+2)(n+3)}{2 \sqrt{n-2}}&=&\frac{n^3}{\sqrt{n}} \left(1+\frac{7}{n}+\frac{37}{2n^2}+O\left(\frac{1}{n^3}\right)\right) .
\end{eqnarray*}
	We change variable $\theta=\frac{\psi}{\alpha_{n}m}$ and we conclude that

\begin{eqnarray*}
{\Gamma}_{\alpha_{n}}^{''}\left(\cos \frac{\psi}{\alpha_{n}m}\right)&=& \frac{C_{\alpha_{n}}^2}{8\pi \sqrt{\pi}}  \bigg\{ \left(\sin \frac{\psi}{2\alpha_n m}\right)^{-7/2}\left(\cos \frac{\psi}{\alpha_{n}m}\right)^{-5/2} \left(\sin  \frac{\psi}{\alpha_n m}\right)^2 \frac{n^2\sqrt{n}}{2}\bigg[-2 \sin \frac{h\psi}{2} \\&& \sin\left(\frac{h\psi}{2}+\psi-\frac{\psi}{2n}+O\left(\frac{\psi}{n^2}\right)+\frac{\pi}{4}\right) +\frac{5}{2}g(n)\cos\left(\psi-\frac{\psi}{2n}+O\left(\frac{\psi}{n^2}\right)-\frac{7}{4}\pi\right)\\&&+O\left(\frac{1}{n}+g(n)^2\right)\bigg]-\sin(\frac{\psi}{2\alpha_{n}m})^{-5/2} \left(\cos \frac{\psi}{2\alpha_{n}m}\right)^{-3/2} \cos \frac{\psi}{\alpha_n m}\\
&&\quad \times n\sqrt{n}\bigg[2\sin \frac{h\psi}{2} \sin\left(\frac{h\psi}{2}+\psi-\frac{\psi}{2n}+O\left(\frac{\psi}{n^2}\right)-\frac{\pi}{4}\right)\\&&+\frac{3}{2}g(n)\cos (\psi-\frac{\psi}{2n}+O\left(\frac{\psi}{n^2}\right)-\frac{5}{4}\pi)-\frac{3}{8}g(n)^2\cos\left(\psi-\frac{\psi}{2n}+O\left(\frac{\psi}{n^2}\right)-\frac{5}{4}\pi\right) \\&&+O\left(\frac{1}{n}\right)+O\left(\frac{g(n)}{n}+g(n)^3\right)\bigg]+O\left(n\sqrt{n}+\psi^2 \sqrt{n}\right) \bigg\}.
\end{eqnarray*}
\end{proof}

\section{Technical results: Asymptotics for the Two-point correlation function}\label{technicalsection}

\subsection{Technical details of the proof of Lemma \ref{asympK}}
In this section we derive the asymptotic expression of all terms appearing in (\ref{K}) of the two-point correlation function $K_n(\psi)$ for large values of $\psi$ uniformly w.r.t. $n$ and $\psi$. 

\begin{proposition}
	For $C < \psi < \alpha_n m\pi/2$, $C>0$,
	as $n\to \infty$,	
	\begin{eqnarray*}
	a&=& -\frac{1}{\pi \psi} \bigg[1-\sin\left(h\psi+2\psi-\frac{\psi}{n}+O\left(\frac{\psi}{n^2}\right)\right)+\frac{6}{\psi} \frac{g(n)}{h}  \sin\left(\frac{h\psi}{2}+\psi-\frac{\psi h}{2n}+O\left(\frac{\psi}{n^2}\right)-\frac{\pi}{4}\right)\\&&\cos\left(\psi-\frac{\psi}{2n}+O\left(\frac{\psi}{n^2}\right)-\frac{5}{4}\pi\right) +\frac{1}{2\pi\psi}+\frac{\cos\left(2h\psi +4\psi-\frac{2\psi}{n}+O\left(\frac{\psi}{n^2}\right)\right)}{2\pi\psi}\\&&+O\left(\frac{1}{\psi^2}+\frac{1}{g(n) n\psi}+\frac{g(n)}{\psi}\right)\bigg]
	\end{eqnarray*}

\end{proposition}
 
 \begin{proof}
From the definition of $a$, given in Section \ref{twopoint}, we have
\begin{eqnarray*}
a&=& -\frac{C_{\alpha_n}^4}{2D_{\alpha_n}} \left( \frac{1}{1-\left(C_{\alpha_n}^2 \sum \frac{2\ell+1}{4\pi} P_\ell\cos (\frac{\psi}{\alpha_nm})\right)^2}\right) 
\left( \sum_{\ell=\alpha_nn}^n \frac{2\ell+1}{4\pi}P_\ell^\prime\left(\cos \left(\frac{\psi}{\alpha_nm}\right)\right) \sin\left(\frac{\psi}{\alpha_nm}\right) \right)^2
\\&=&-\frac{1}{2D_{\alpha_n}} \left( \frac{1}{1-\left(\Gamma_{\alpha_n}(\cos \frac{\psi}{\alpha_nm})\right)^2}\right) 
\left(\Gamma^\prime_{\alpha_{n}}\left(\cos \frac{\psi}{\alpha_nm}\right)\right)^2\\
& =& -\frac{1}{2D_{\alpha_n}} \left(\Gamma_{\alpha_{n}}^\prime\left(\cos \left(\frac{\psi}{\alpha_nm}\right)\right) \right)^2 \left(1+\Gamma_{\alpha_{n}}^2\left(\cos \left(\frac{\psi}{\alpha_nm}\right)\right)+O\left(\Gamma_{\alpha_{n}}^4\left(\cos \left(\frac{\psi}{\alpha_nm}\right)\right)\right)\right).
\end{eqnarray*}
Using the expansion of the covariance function and its derivatives found in the previous section, we get
\begin{eqnarray*}
	a&=&\frac{-2n^2}{2D_{\alpha_n}\pi\psi}\bigg\{  \left(\frac{1}{2}-\frac{1}{2} \sin\left(h\psi+2\psi-\frac{\psi}{n}+O\left(\frac{\psi}{n^2}\right)\right)\right) 
\\&&+\frac{3g(n)}{\psi h } \sin\left(\frac{h\psi}{2}+\psi-\frac{\psi}{2n}+O\left(\frac{\psi}{n^2}\right)-\frac{\pi}{4}\right)\cos\left(\psi-\frac{\psi}{2n}+O\left(\frac{\psi}{n^2}\right)-\frac{5}{4}\pi\right) \\&&
+ O\left(\frac{g(n)}{\psi}\right)+O\left(\frac{1}{\psi n}\right)+O\left(\frac{1}{\psi^2}\right)+O\left(\frac{\sqrt{\psi}\psi}{n^4\sqrt{n}g(n)}\right)\bigg\}
\\&&\quad \times
\bigg\{1+\frac{2}{\pi\psi} \bigg[\frac{1}{2}+\frac{\sin\left(h\psi+2\psi-\frac{\psi}{n}+O\left(\frac{\psi}{n^2}\right)\right)}{2}+\frac{g(n)}{h\psi} \sin\left(\frac{h\psi}{2}+\psi-\frac{\psi}{2n}+O\left(\frac{\psi}{n^2}\right)+\frac{\pi}{4}\right)\\&&\cos\left(\psi-\frac{\psi}{2n}+O\left(\frac{\psi}{n^2}\right)-\frac{3}{4}\pi\right)+O\left(\frac{1}{\psi^2}\right)+O\left(\frac{1}{g(n)n\psi}+\frac{g(n)}{\psi}\right) \bigg\}.
\end{eqnarray*}

Multiplying and using the fact that $\sin^2(\alpha)= \frac{1-\cos(2\alpha)}{2}$ and that 
\begin{eqnarray*}
D_{\alpha_{n}}&=& C_{\alpha_n}^2 \sum_{\ell=\alpha_nn}^{n} \frac{\ell(\ell+1)}{2} \frac{2\ell+1}{4\pi}= \frac{4\pi}{2 n^2 g(n)} \sum_{k=0}^{\infty} \left(\frac{g(n)}{2}-\frac{1}{ng(n)}-\frac{1}{2n^2g(n)}\right)^k \frac{1}{8\pi} \frac{n^44g(n)}{2}\\
&&\left[1- \frac{6}{4} g(n)+g(n)^2-\frac{g(n)^3}{4} +\frac{1}{ng(n)}+\frac{6}{4g(n)n^2}+\frac{2}{4g(n)n^3}+\frac{g(n)}{4n^2}-\frac{2 }{4n^2} \right]\\
&=& \frac{n^2}{2}\left[1-g(n)-\frac{1}{g(n)^2n^2}+O\left(\frac{1}{n}+g(n)^2+\frac{1}{n^2g}\right)\right]\\&=& \frac{n^2}{2}\left[1+O(g(n))+O\left(\frac{1}{g(n)^2n^2}\right)+O\left(\frac{1}{n}\right)\right],
\end{eqnarray*}
 we get the thesis of the proposition.
\end{proof}

\begin{proposition} 	For $C < \psi < \alpha_n m\pi/2$, $C>0$,
	as $n\to \infty$,	
	\begin{eqnarray*}
b&=&  \sqrt{\frac{2}{\pi\psi}}\bigg[  \sin\left(\frac{h\psi}{2}+\psi-\frac{\psi}{2n}+O\left(\frac{\psi}{n^2}\right) +\frac{\pi}{4}\right)-\frac{5 g(n)}{2\psi h} \cos\left(\psi-\frac{\psi}{2n}+O\left(\frac{\psi}{n^2}\right) -\frac{7}{4}\pi\right)
	\\&&+\frac{1}{{\psi}}\sin\left(\frac{h\psi}{2}+\psi-\frac{\psi}{2n}+O\left(\frac{\psi}{n^2}\right) -\frac{\pi}{4}\right) -\frac{1}{\pi\psi} \sin\left(\frac{h\psi}{2}+\psi-\frac{\psi}{2n}+O\left(\frac{\psi}{n^2}\right) +\frac{\pi}{4}\right)+O\left(\frac{1}{\psi^2}\right)\\&&
	+\frac{1}{\pi \psi} \sin\left(h\psi+2\psi-\frac{\psi}{n}+O\left(\frac{\psi}{n^2}\right)\right) \sin\left(\frac{h\psi}{2}+\psi-\frac{\psi}{2n}+O\left(\frac{\psi}{n^2}\right) +\frac{\pi}{4}\right) 	\\&&+O\left(\frac{1}{ng(n)} +\frac{g(n)}{\psi}+\frac{\sqrt{\psi}}{ng(n)\sqrt{n}} \right)
+O\left(\frac{1}{ng(n)\psi}\right) +O(g(n)) \bigg].
	\end{eqnarray*}
\end{proposition}

\begin{proof}
	As we did in the previous proposition, starting from the definition of $b$, given in Section \ref{twopoint}, we have
\begin{eqnarray*}
b &=& \frac{C_{\alpha_n}^2}{2D_{\alpha_n}} \bigg[ \sum_{\ell=n\alpha_n}^{n} \frac{2\ell+1}{4\pi}P_\ell^\prime \left( \cos \left(\frac{\psi}{\alpha_nm}\right) \right)\left(\cos \left(\frac{\psi}{\alpha_nm}\right)\right)-\frac{2\ell+1}{4\pi}P_\ell^{''}\left(\cos \left(\frac{\psi}{\alpha_nm}\right) \right) \sin\left( \frac{\psi}{\alpha_nm}\right)^2\\&& - \frac{\left( C^2 \sum \frac{2\ell+1}{4\pi}P_\ell^\prime\left(\cos\left(\frac{\psi}{\alpha_nm}\right)\right) \sin\left(\frac{\psi}{\alpha_nm}\right)\right)^2}{1-\left(C^2 \sum_{\ell=\alpha_{n}n}^{n} \frac{2\ell+1}{4\pi}P_\ell\left(\cos \left(\frac{\psi}{\alpha_nm}\right)\right)\right)^2} \sum_{\ell=n\alpha_n}^{n} \frac{2\ell+1}{4\pi}P_\ell \left(\cos \left(\frac{\psi}{\alpha_nm}\right)\right)\bigg] 
\\&=&\frac{1}{2D_{\alpha_n}} \left[ -\Gamma_{\alpha_{n}}''\left(\cos \left(\frac{\psi}{\alpha_nm}\right)\right)-\frac{\Gamma_{\alpha_{n}}\left(\cos \left(\frac{\psi}{\alpha_nm}\right)\right)}{1-\Gamma_{\alpha_{n}}\left(\cos \left(\frac{\psi}{\alpha_nm}\right)\right)^2} \Gamma_{\alpha_{n}}^{\prime}\left(\cos\left(\frac{\psi}{\alpha_nm}\right)\right)^2 \right]\\&=&\frac{1}{2D_{\alpha_n}} \bigg[ -\Gamma_{\alpha_{n}}''\left(\cos \left(\frac{\psi}{\alpha_nm}\right)\right)-\Gamma_{\alpha_{n}}\left(\cos \left(\frac{\psi}{\alpha_nm}\right)\right) \left(\Gamma_{\alpha_{n}}^{\prime}\left(\cos\left(\frac{\psi}{\alpha_nm}\right)\right)\right)^2 \\&& \quad \times  \left(1+\Gamma_{\alpha_{n}}^2\left(\cos\left(\frac{\psi}{\alpha_nm}\right)\right)+O\left(\Gamma_{\alpha_{n}}^4\left(\cos\left(\frac{\psi}{\alpha_nm}\right)\right)\right)\right)\bigg].
\end{eqnarray*}

From corollary \ref{gamma-primo-quadro} and corollary \ref{gamma} we derive that

\begin{eqnarray*}
&&\Gamma_{\alpha_{n}}\left(\cos \frac{\psi}{\alpha_{n}m}\right) \Gamma_{\alpha_{n}}'^2\left(\cos \frac{\psi}{\alpha_{n}m}\right)= \frac{2n^2}{\pi\psi} \bigg\{ \frac{1}{2} -\frac{1}{2}\sin\left(h\psi+2\psi-\frac{\psi}{n}+O\left(\frac{\psi}{n^2}\right)  \right) \\ && \qquad+\frac{3}{\psi} \sin\left(\frac{h\psi}{2}+\psi-\frac{\psi}{2n}+O\left(\frac{\psi}{n^2}\right)-\frac{\pi}{4}\right) \cos\left(\psi-\frac{\psi}{2n}+O\left(\frac{\psi}{n^2}\right)-\frac{5}{4}\pi\right) +O\left(\frac{g}{\psi}+\frac{1}{\psi^2}+\frac{1}{n\psi}\right) \bigg\} \\&& \qquad 
\times \frac{\sqrt{2}}{\sqrt{\pi\psi}} \frac{h}{g(n)}\bigg\{  \sin\left(\frac{h\psi}{2}+\psi-\frac{\psi}{2n}+O\left(\frac{\psi}{n^2}\right)+\frac{\pi}{4}\right)+\frac{g(n)\cos\left(\psi-\frac{\psi}{2n}+O\left(\frac{\psi}{n^2}\right)-\frac{3}{4}\pi\right)}{2\psi h }\\&&\quad +O\left(\frac{1}{ng(n)}+g(n) \right) \bigg\}
\\&&= \frac{2n^2}{\pi\psi} \sqrt{\frac{2}{\pi\psi}} \frac{h}{g(n)} \bigg\{ \frac{1}{2} \sin\left(\frac{h\psi}{2}+\psi-\frac{\psi}{2n}+O\left(\frac{\psi}{n^2}\right)+\frac{\pi}{2} \right)+\frac{g(n)}{4 \psi h} \cos\left(\psi-\frac{\psi}{2n}+O\left(\frac{\psi}{n^2}\right) -\frac{3}{4}\pi\right) \\&& \qquad -\frac{1}{2} \sin\left(h\psi+2\psi-\frac{\psi}{n}+O\left(\frac{\psi}{n^2}\right)\right)\sin\left(\frac{h\psi}{2}+\psi-\frac{\psi}{2n}+O\left(\frac{\psi}{n^2}\right)+\frac{\pi}{4} \right) \\&& \qquad -\frac{g(n)}{4\psi h} \cos\left(\psi-\frac{\psi}{2n}+O\left(\frac{\psi}{n^2}\right) -\frac{3\pi}{4}\right)\sin\left(h\psi+2\psi-\frac{\psi}{n}+O\left(\frac{\psi}{n^2}\right) \right) \\&& \qquad +\frac{3g(n)}{\psi h} \sin\left(h\psi/2+\psi-\frac{\psi}{n}+O\left(\frac{\psi}{n^2}\right) -\frac{\pi}{4}\right) \sin \left (\frac{h\psi}{2}
+\psi-\frac{\psi}{n}+O\left(\frac{\psi}{n^2}\right)+\frac{\pi}{4} \right) \\&& \qquad \cos\left(\psi-\frac{\psi}{2n}+O\left(\frac{\psi}{n^2}\right)-\frac{5}{4}\pi \right)+O\left(\frac{1}{n\psi} +\frac{g(n)}{\psi^2}+\frac{g(n)^2}{\psi}+\frac{1}{\psi^2g(n)n}\right) \bigg\} .
\end{eqnarray*}

It follows that

\begin{eqnarray*}
b&=& \frac{1}{2D_{\alpha_n}} \sqrt{\frac{2}{\pi\psi}} n^2 \frac{h}{g(n)}\bigg[  \sin\left(\frac{h\psi}{2}+\psi-\frac{\psi}{2n}+O\left(\frac{\psi}{n^2}\right) +\frac{\pi}{4}\right)-\frac{5 g(n)}{2\psi h} \cos\left(\psi-\frac{\psi}{2n}+O\left(\frac{\psi}{n^2}\right) -\frac{7}{4}\pi\right)
\\&&+\frac{1}{{\psi}}\sin\left(\frac{h\psi}{2}+\psi-\frac{\psi}{2n}+O\left(\frac{\psi}{n^2}\right) -\frac{\pi}{4}\right) -\frac{1}{\pi\psi} \sin\left(\frac{h\psi}{2}+\psi-\frac{\psi}{2n}+O\left(\frac{\psi}{n^2}\right) +\frac{\pi}{4}\right)+O\left(\frac{1}{\psi^2}\right)\\&&
+\frac{1}{\pi \psi} \sin\left(h\psi+2\psi-\frac{\psi}{n}+O\left(\frac{\psi}{n^2}\right)\right) \sin\left(\frac{h\psi}{2}+\psi-\frac{\psi}{2n}+O\left(\frac{\psi}{n^2}\right) +\frac{\pi}{4}\right) +O\left(\frac{1}{ng(n)\psi} +\frac{g(n)}{\psi}\right)
\\&&+O\left(\frac{1}{ng(n)}\right) \bigg],
\end{eqnarray*}
from which we get the thesis of the proposition.
\end{proof}

\begin{corollary}
	For $C < \psi < \alpha_n m\pi/2$, $C>0$,
as $n\to \infty$,	

\begin{eqnarray*}
b^2&=& \frac{1}{\pi\psi} \bigg\{ 1+\sin\left(h\psi+ 2\psi-\frac{\psi}{2n} +O\left(\frac{\psi}{n^2}\right)\right) -\frac{5}{\psi}  \cos \left(\frac{h \psi}{2}+ \psi+\frac{\psi}{2n} +O\left(\frac{\psi}{n^2}\right)\right) \\&&
 -\frac{5}{\psi} \sin\left(\frac{h\psi}{2}\right)-\frac{2}{\psi} \cos\left(h\psi +2\psi +\frac{\psi}{2n} +O\left(\frac{\psi}{n^2}\right)\right) -\frac{1}{\pi \psi} -\frac{1}{\pi\psi} \cos\left(2h\psi+4\psi +\frac{\psi}{n} +O\left(\frac{\psi}{n^2}\right)\right)\\&&
+O\left(\frac{g(n)}{\psi}\right)+O\left(\frac{1}{\psi^2}+\frac{1}{ng(n)\psi}\right)\bigg\},\\
b^4 &=& \frac{1}{\pi^2 \psi^2} \bigg\{ \frac{3}{2} -\frac{\cos\left(2h\psi+4\psi +\frac{\psi}{n} +O\left(\frac{\psi}{n^2}\right)\right)}{2}+2\sin\left(h\psi+2\psi +\frac{\psi}{2n} +O\left(\frac{\psi}{n^2}\right)\right)+O\left(\frac{1}{\psi}\right)\bigg\} ,\\
a^2&=& \frac{1}{\pi^2 \psi^2} \bigg\{ \frac{3}{2} -\frac{\cos\left(2h\psi+4\psi +\frac{\psi}{n} +O\left(\frac{\psi}{n^2}\right)\right)}{2}-2\sin\left(h\psi+2\psi +\frac{\psi}{2n} +O\left(\frac{\psi}{n^2}\right)\right)+O\left(\frac{1}{\psi}\right)\bigg\} ,\\
a b^2 &=& -\frac{1}{\pi^2 \psi^2} \bigg\{ \frac{1}{2} +\frac{\cos\left(2h\psi+4\psi +\frac{\psi}{n} +O\left(\frac{\psi}{n^2}\right)\right)}{2}+O\left(\frac{1}{\psi}\right)\bigg\} ,
\end{eqnarray*}
and 
\begin{eqnarray*}
a \Gamma_{\alpha_{n}}^2\left(\cos \frac{\psi}{\alpha_{n}m}\right)\hspace{-3mm}&=&\hspace{-3mm}-\frac{1}{\pi^2 \psi^2} \bigg\{ \frac{1}{2} +\frac{\cos\left(2h\psi+4\psi +\frac{\psi}{n} +O\left(\frac{\psi}{n^2}\right)\right)}{2}+O\left(\frac{1}{\psi}\right)\bigg\} , \\
b^2 \Gamma_{\alpha_{n}}^2\left(\cos \frac{\psi}{\alpha_{n}m}\right)\hspace{-3mm}&=&\hspace{-3mm}\frac{1}{\pi^2 \psi^2} \bigg\{ \frac{3}{2} -\frac{\cos\left(2h\psi+4\psi +\frac{\psi}{n} +O\left(\frac{\psi}{n^2}\right)\right)}{2}+2\sin\left(h\psi+2\psi +\frac{\psi}{2n} +O\left(\frac{\psi}{n^2}\right)\right)+O\left(\frac{1}{\psi}\right)\bigg\} ,\\
\Gamma_{\alpha_{n}}^4\left(\cos \frac{\psi}{\alpha_{n}m}\right)\hspace{-3mm}&=&\hspace{-3mm}\frac{4}{\pi^2 \psi^2} \bigg\{ \frac{3}{8} -\frac{\cos\left(2h\psi+4\psi +\frac{\psi}{n} +O\left(\frac{\psi}{n^2}\right)\right)}{8}+\frac{1}{2}\sin\left(h\psi+2\psi +\frac{\psi}{2n} +O\left(\frac{\psi}{n^2}\right)\right)+O\left(\frac{1}{\psi}\right)\bigg\}. 
\end{eqnarray*}

\end{corollary}

\subsection{ Proof of Proposition \ref{small theta}}\label{sectionsmalltheta}
Proposition \ref{small theta} is a readaptation of Lemma 4.4 \cite{ST} (see also Lemma 3.4 and corollary 3.5 \cite{W09}) to our context. From the mentioned papers we have that 
$$K_{\alpha_{n}}(x,N) \ll \frac{D_{\alpha_{n}}}{\sqrt{1-(C_{\alpha_{n}}^2 \sum_{\ell=\alpha_{n} }^{n} \frac{2\ell+1}{4\pi}P_\ell(\cos \theta))^2}}.$$

Then 

$$\int_{0}^{C/(\alpha_{n}m)} \bigg|K_{\alpha_{n}}(\theta)-\frac{1}{4}\bigg| \sin \theta \, d\theta $$ is bounded by

$$\int_{0}^{C/(\alpha_{n}m)}  \frac{D_{\alpha_{n}}}{\sqrt{1-(C_{\alpha_{n}}^2 \sum_{ \ell=\alpha_{n}}^{n} \frac{2\ell+1}{4\pi}P_\ell(\cos \theta))^2}} \sin (\theta) \, d\theta.$$

From the Taylor approximation $P_\ell(\cos\theta)=1-\theta^2\frac{\ell(\ell+1)}{2}+\theta^4O(\ell^4)$, we conclude that
\begin{eqnarray*}
	&&1-P_\ell(\cos\theta)P_{\ell'}(\cos\theta)\\&&=1-\bigg(1-\theta^2 \frac{\ell(\ell+1)}{2}+O(\ell^4 \theta^4) -\theta^2 \frac{\ell'(\ell'+1)}{2}+O(\ell'^4 \theta^4)+ \theta^4 \frac{\ell'(\ell'+1)}{2}\frac{\ell(\ell+1)}{2}\bigg)\\&&=\theta^2 \frac{\ell(\ell+1)}{2} +\theta^2 \frac{\ell'(\ell'+1)}{2}+ \theta^4 \frac{\ell'(\ell'+1)}{2}\frac{\ell(\ell+1)}{2}+O(\ell^4 \theta^4)+O(\ell'^4 \theta^4),
\end{eqnarray*}
and thus,
\begin{eqnarray*}
1-\left(C_{\alpha_{n}}^2 \sum_{\ell=\alpha_{n} }^{n} \frac{2\ell+1}{4\pi}P_\ell(\cos \theta)\right)^2&=&C_{\alpha_{n}}^4 \sum_{\ell=\alpha_{n} }^{n} \sum_{\ell^\prime =\alpha_{n}}^{n}  \frac{2\ell+1}{4\pi}\frac{2\ell'+1}{4\pi}(1-P_\ell(\cos \theta) P_\ell'(\cos \theta)) 
\\&=& \theta^2 \bigg[2C_{\alpha_{n}}^2 \sum_{ \ell=\alpha_{n}}^{n} \frac{2\ell+1}{4\pi} \frac{\ell(\ell+1)}{2} \sum_{ \ell'=\alpha_{n}}^{n} C_{\alpha_{n}}^2 \frac{2\ell'+1}{4\pi} \\&&+\theta^2 C_{\alpha_{n}}^4 \sum_{\ell,\ell^\prime=\alpha_{n}}^{n} \frac{2\ell+1}{4\pi}\frac{2\ell'+1}{4\pi} \frac{\ell(\ell+1)}{4}\frac{\ell'(\ell'+1)}{4}
\\&&+O\left(\theta^2\ell^4C_{\alpha_{n}}^4 \sum \frac{2\ell+1}{4\pi}\frac{2\ell'+1}{4\pi} \right)\bigg]
\\&=& 2\theta^2D_{\alpha_n}\left[1+O(\theta^2 D_{\alpha_n})\right] .
\end{eqnarray*}

This implies that 
$$\int_{0}^{C/(\alpha_{n}m)} \frac{D_{\alpha_{n}}}{\sqrt{1-u^2}}\sin \theta \,d\theta= O\left(\frac{D_{\alpha_n}}{\sqrt{D_{\alpha_n}}} \int_{0}^{C/(\alpha_{n}m)} \frac{1}{\sqrt{1+O(D_{\alpha_n}\theta^2)}} \frac{\sin \theta}{\theta} \right)=O\left(1\right).$$

\end{document}